\documentclass[10pt,a4paper,twoside]{amsart}
\usepackage[T1]{fontenc}
\usepackage[utf8]{inputenc}
\usepackage[english]{babel}
\usepackage{amsmath}
\usepackage{amsthm}
\usepackage{amssymb}
\usepackage{amsfonts}
\usepackage{amsxtra}
\usepackage{enumerate}
\usepackage{mathrsfs}
\usepackage{ulem}
\usepackage{verbatim}
\usepackage{color}
\usepackage[pagewise,displaymath,mathlines]{lineno}

\newcommand{\N}{\mathbb N}

\newcommand{\Q}{\mathbb Q}
\newcommand{\R}{\mathbb R}

\newtheorem*{claim}{Claim}

\newtheorem*{examp}{Example}
\newtheorem{theorem}{Theorem}[section]
\newtheorem{lemma}[theorem]{Lemma}
\newtheorem{cor}[theorem]{Corollary}

\newtheorem{prop}[theorem]{Proposition}

\theoremstyle{remark}
\newtheorem{remark}[theorem]{Remark}

\theoremstyle{definition}
\newtheorem{definition}[theorem]{Definition}

\DeclareMathOperator*{\supp}{supp}
\DeclareMathOperator*{\conv}{conv}

\DeclareMathOperator*{\dist}{dist}

\DeclareMathOperator*{\inte}{int}
\DeclareMathOperator*{\inj}{inj}

\DeclareMathOperator*{\id}{id}
\DeclareMathOperator*{\ev}{ev}
\DeclareMathOperator*{\tev}{[\partial_t ev]}
\DeclareMathOperator*{\EV}{ev}

\let\phi=\varphi
\let\e=\varepsilon

\begin{document}

\setcounter{page}{1}

\title[Optimal transportion for Lorentzian cost functions]{Theory of optimal transport for Lorentzian cost functions}

\author[Stefan Suhr]{Stefan Suhr}
\address{Fakult\"at f\"ur Mathematik, Ruhr-Universit\"at Bochum, Universit\"atsstra\ss e 150, 44780 Bochum, Germany}
\email{stefan.suhr@rub.de}

\date{\today}

\begin{abstract}
The optimal transport problem in the context of Lorentz-Finsler geometry is studied. Besides 
deducing the existence of optimal couplings a result on the intermediate regularity of optimal couplings 
is given. One further establishes a solution to the Monge problem and an exact criterion for the 
existence of causal couplings. The results generalize parts of \cite{berpue}, \cite{br2} and \cite{eckmil}. 
\end{abstract}

\maketitle

\section{Introduction}
This article studies optimal transportation in Lorentz-Finsler manifolds from the geometric point of view. The geometric viewpoint necessitates that one passes
to a spacetime as configuration space. In a spacetime the time parameter is part of the geometry. There are multiple choices for the time parameter and, as is well 
known in Lorentzian geometry, no choice is preferred. In other words it is not canonical which part of the spacetime is space, or equivalently which points are 
isochronous. Thus isochronicity in Lorentz-Finsler geometry is subject to a choice. 
Usually this choice is made via singling out a time function whose level sets are then thought of as constituting space. After choosing a time function a transport 
problem can be posed between two level sets of this time function. Solutions to such transport problems are provided for example in \cite{berpue}. Transport 
problems originating in applications to relativity though, such as the early universe reconstruction problem (see below), cannot be brought into such a form
thus motivating the spacetime perspective. This follows from the fact that the support of one martingale is the boundary of the causal past of a set. In this case 
the support is in general not acausal (see Section \ref{results}). Sets that are not acausal are never the level set of a time function. 
In the formulation of other transport problems the measures are not concentrated on a single level set of a time function, but are rather
distributed on a continuum of level sets, i.e. distributed over a continuum of time parameters. This appears for example when considering the same transport 
problem relative to two different time functions. For a discussion of the dependence on different time functions see  \cite{miller17}.

The first one to take notice of the problem of optimal transportation in Lorentzian geometry was 
\cite{brenier1}. Therein a transportation problem is proposed, which only weakly disguised is the 
problem of transportation between parallel spacelike hyperplanes in Minkowski space with respect to the 
negative Lorentzian distance extended by $\infty$. Here a strong form of isochronicity
is assumed for the support of each measure, i.e. being supported on level sets of a linear 
time function. Following this formulation \cite{berpue} generalized the problem to a wider class of 
functions called {\it relativistic costs}, and gave inter alia a solution to the Monge problem while 
staying in the same basic geometric frame.

The {\it early universe reconstruction problem}, studied in \cite{br2} and \cite{fr} with methods of optimal transportation, asks whether one can construct 
the trajectories of masses from the big bang to their present day positions in Robertson-Walker spacetimes. A mathematical formulation for general 
globally hyperbolic spacetimes would read as follows: Given two measures, one concentrated on a Cauchy hypersurface, the other on the past cone 
of a point. Then what can be said about the trajectories of the minimizers in a dynamical optimal coupling (see Definition \ref{defdynoptcou}) of the two measures? 
\cite{fr} gives a justification to why the problem can be studied with methods from optimal transportation.

The first question that comes to mind when studying a cost function which take an infinite value, such as the cost function considered here, is whether 
there exists a coupling of two given measures with finite cost. This problem was studied in \cite{berpue}. 
Recently in a more systematic approach \cite{eckmil}, \cite{mil} and \cite{eckmil2} have studied the 
the problem and the causal evolution of measures in Lorentzian geometry. Theorem \ref{P2} extends the existing results on the question  
to a more abstract setting including metric spaces. 

The other results in this article generalize the previous approaches to the problem of Lorentzian optimal 
transportation in two directions: The first goal as already mentioned 
above is the structure of the support of the measures involved, i.e. passing from being supported on 
surfaces of isochronicity (level sets of time functions) to being distributed in space and time or on achronal sets, see
Theorem \ref{intermediateregularity}, \ref{lipreg}, \ref{Thmmonge}, and \ref{Thmmonge2}. From the physics point of view 
this means that observations are not only made at a single point in time but rather over a stretch
of time or cannot be brought into the form of a single time parameter. 

Second there is the extension to Lorentz-Finsler geometry. This category includes Lorentzian geometry. 
Thus one can now study transport problems in relativity in their full generality. The step from 
relativistic cost functions and Robertson-Walker spacetimes, respectively, to globally hyperbolic Lorentz-Finsler 
spacetimes is comparable with passing from Euclidian space to Riemannian manifolds in the theory of 
optimal transport. 

The study was motivated by these comments and the prospect of future developments mimicking the 
relations between optimal transport and fluid dynamics, Riemannian geometry and formulations of 
synthetic curvature. The article gives besides these generalizations new results on the structure
of optimal couplings not known even for relativistic cost functions. Section \ref{results} describes 
the setting and states the main results. Section \ref{proofs} provides the proofs.

{\it Acknowledgement:}
The author would like to thank Victor Bangert for suggesting the problem of optimal transportation in the context of Lorentzian geometry and Albert 
Fathi for encouraging the pursuit of the project. The author would further like to thank Patrick Bernard for providing the opportunity to carry out the ideas 
for this article and making numerous suggestions which helped to shape the exposition of the present results. Valentine Roos and Rodolfo R\'ios-Zertuche 
are kindly thanked for many helpful discussions in the process of this research.

\section{The results}\label{results}

Let $M$ be a smooth manifold of dimension $m\ge 2$. Denote by $\pi_{TM}\colon TM\to M$ the canonical 
projection of tangent vectors to their base point. Throughout the article one fixes a 
complete Riemannian metric $h$ on $M$. The norm $|.|$ and distances $\dist(.,.)$ are understood to be
induced by $h$, unless noted otherwise. Recall that $h$ induces a Riemannian metric on $TM$. Distances in $TM$ are understood 
to be induced by this metric. The metric ball around $y$ with radius $r$ is denoted by $B_r(y)$.
Set $T^0M$ to be the image of the zero section of $TM$ and $0_p$ the zero vector in $TM_p$. 

Consider a continuous function $\mathbb{L}\colon TM\to \R$ smooth on $TM\setminus T^0M$ and positive homogenous of degree $2$ such that the second 
fiber derivative is nondegenerate with index $m-1$. One says that  $\mathcal{C}\subset TM$ is a {\it closed cone field} if $\mathcal{C}_p:=
\mathcal{C}\cap TM_p$ is a closed convex cone for all $p\in M$ and $\mathcal{C}\cup T^0M$ is a closed subset of $TM$.
A {\it causal structure} $\mathcal{C}$ of $(M,\mathbb{L})$ is then a choice of a closed cone field $\mathcal{C}$ with $\pi_{TM}(\mathcal{C}) =M$ such that 
$\inte\mathcal{C}$, the open interior of $\mathcal{C}$, is a connected component of $\{\mathbb{L}>0\}$. 
For every point $p\in M$ every connected component of $TM_p\cap \{\mathbb{L}> 0\}\subset 
TM_p$ belongs to a unique causal structure up to a finite cover, see 
Section \ref{s3.1}.

Fix a causal structure $\mathcal{C}$ for $(M,\mathbb{L})$. Define a new Lagrangian $L$ on $TM$ by setting 
$$L(v):=\begin{cases}
-\sqrt{\mathbb{L}(v)},& \text{ for }v\in \mathcal{C},\\
\infty,&\text{ otherwise.}
\end{cases}$$
The function $L$ is fiberwise convex, finite on its domain and positive homogeneous of degree one. It further is smooth 
on $\inte\mathcal{C}$. The function $L$ has the features of a Finsler metric of Lorentzian type. This justifies to call 
the pair $(M,L)$ a {\it Lorentz-Finsler manifold}.  The generality of Lorentz-Finsler geometry is chosen in view of 
recent developments in the area, see e.g. \cite{jasa}, \cite{minguzzi141}, \cite{minguzzi142}, \cite{minguzzi143}, and the goal 
to achieve a scope comparable to the one of Tonelli-Lagrangian systems, see e.g. \cite{bebu1}, \cite{bebu2} and \cite{fafi}. 

One calls an absolutely continuous curve $\gamma\colon I\to M$ {\it ($\mathcal{C}$-)causal} if $\dot\gamma\in \mathcal{C}$ 
whenever the tangent vector exists. A causal curve $\gamma\colon I\to M$ is {\it timelike} if for all $s\in I$ there exists $\e,\delta>0$ 
such that $\dist(\dot\gamma(t),\partial\mathcal{C})\ge \e |\dot\gamma(t)|$,  for every $t\in I$ for which $\dot\gamma(t)$ exists and 
$|s-t|<\delta$.  

Denote by $J^+(p)$ the set of points $q\in M$ such that there exists a causal curve with initial point $p$ and 
terminal point $q$. $J^-(p)$ is the set of points $q\in M$ such that there exists a causal curve with initial point $q$ and terminal point $p$.
$I^\pm(p)$ are defined in a similar way where causal curves are replaced by timelike ones.
For $A\subset M$ set $J^\pm(A) :=\cup_{p\in A} J^\pm (p)$. Define the set
$$J^+:=\{(p,q)\in M\times M|\; q\in J^+(p)\}.$$ 
For an open set $U\subseteq M$ define $J^\pm_U$ and $I^\pm_U$ as before for the restriction $(U,\mathcal{C}|_U)$, with 
$\mathcal{C}|_U:=TU\cap \mathcal{C}$.

A Lorentz-Finsler manifold is said to be {\it causal} if it does not admit a closed causal curve. 

\begin{definition}\label{D1}
A causal Lorentz-Finsler manifold $(M,L)$ is {\it globally hyperbolic} if the sets $J^+(p)\cap J^-(q)$ are compact for all $p,q\in M$. 
\end{definition}

Every causal structure is a closed and nondegenerate cone field which is wider than an open 
nondegenerate cone field in the sense of \cite{suhr1}. Therefore Definition \ref{D1} implies that the causal structure of a 
globally hyperbolic Lorentz-Finsler manifold is globally hyperbolic in the sense of \cite{suhr1}. 
By \cite[Theorem 3]{suhr1} there exists a smooth function $\tau\colon M\to \R$ 
(called a {\it splitting}) with 
$$-d\tau(v) \le \min \{L(v),-|v|\}$$
for all $v\in \mathcal{C}$. 
\cite[Corollary 1.8]{suhr1} implies that there exists a diffeomorphism (also called a 
{\it splitting}) $M\cong \R\times N$ such that 
$$\tau\colon M\cong\R\times N\to \R,\; p\cong (\theta,x)\mapsto \theta$$ 
if $(M,L)$ is globally hyperbolic. Note that $\tau$ is by far not unique.

\begin{remark}\label{remarkinextendable}
For a causal curve $\gamma\colon (a,b)\to M$ which leaves every compact subset of $M$ for both $t\downarrow a$ and $t\uparrow b$, one has $\lim_{t\downarrow a} 
\tau\circ \gamma(t)=-\infty$ and $\lim_{t\uparrow b} \tau\circ \gamma(t)=\infty$. This follows from the completeness of $h$.
\end{remark}

Define the Lagrangian action relative to $L$:
$$\mathcal{A}(\gamma):=\begin{cases}
\int L(\dot\gamma)dt,&\text{ if $\gamma$ is $\mathcal{C}$-causal,}\\
\infty,&\text{ else.}
\end{cases}$$
A causal curve $\gamma\colon I\to M$ is an {\it $\mathcal{A}$-minimizer} between its endpoints $p,q\in M$ if 
$$\mathcal{A}(\gamma)=\inf\{\mathcal{A}(\eta)|\; \eta\text{ connects $p$ and $q$}\}.$$

\begin{prop}\label{P0}
Let $(M,L)$ be globally hyperbolic. Then for every pair of points $p,q\in M$ with $(p,q)\in J^+$ there exists an $\mathcal{A}$-minimizer $\gamma\colon I\to M$ 
with finite action connecting the two points. The minimizer $\gamma$ solves the Euler-Lagrange equation of $\mathbb{L}$ up to monotone 
reparameterization and one has $\dot\gamma\in\mathcal{C}$ whenever the tangent vector exists.
\end{prop}

Since the arguments are completely analogous to the Lorentzian case only a brief summary of the proof is given for completeness. 
\begin{proof}[Sketch of proof]
For $(p,q)\in J^+$ consider the space $\mathcal{C}(p,q)$ of causal curves $\eta$
from $p$ to $q$ with $\tau\circ\eta \equiv \id$ and equipped with the uniform
$C^0$-topology. $\mathcal{C}(p,q)$ is nonempty and compact  since $(M,\mathcal{C})$ is globally hyperbolic by \cite[Proposition 5.15]{suhr1}. 

The restriction $\mathcal{A}\colon \mathcal{C}(p,q)\to \R$ is lower semicontinuous. Therefore there exists an $\mathcal{A}$-minimizer 
$\gamma\colon [\tau(p),\tau(q)]\to M$ in $\mathcal{C}(p,q)$ with $\dot\gamma\in \mathcal{C}$ whenever the tangent exists. 
Now by \cite[Theorem 6]{minguzzi142} $\gamma$ is a pregeodesic, i.e. solves the Euler-Lagrange equations of the action functional associated 
to $\mathbb{L}$ up to a monotone reparameterization.
\end{proof}

For a globally hyperbolic Lorentz-Finsler manifold define the {\it cost function relative to} $L$:
\begin{align*}
c_L\colon M\times M&\to \R\cup\{\infty\}\\
(p,q)&\mapsto\inf\left\{\left.\mathcal{A}(\gamma)\right|\; \gamma\text{ connects $p$ and $q$}\right\}
\end{align*}
$c_L$ satisfies the triangle inequality 
$$c_L(p,r)\le c_L(p,q)+c_L(q,r)$$
for all $p,q,r\in M$. The inequality is nontrivial only if $(p,q),(q,r)\in J^+$. In this case the inequality follows from the observations that the causal relation $J^+$ is transitive 
and $c_L$ is defined via an infimum. For an $\mathcal{A}$-minimizer $\gamma\colon [a,c]\to M$ and $a\le b\le c$ one has 
$$c_L(\gamma(a),\gamma(c))=c_L(\gamma(a),\gamma(b))+c_L(\gamma(b),\gamma(c)).$$

For two Borel probability measures $\mu,\nu$ on $M$ call a Borel probability measure $\pi$ on $M\times M$ a {\it coupling of $\mu$ and $\nu$} if 
$(p_1)_\sharp \pi=\mu$ and $(p_2)_\sharp \pi=\nu$ where $p_{1},p_2\colon M\times M\to M$ are the  projections onto the first and second factor. Recall that
the push-forward $(p_i)_\sharp \pi$ is defined as $(p_i)_\sharp \pi (A):=\pi (p_i^{-1}(A))$. The set of couplings of $\mu$ and $\nu$ is denoted by $\Pi(\mu,\nu)$.

The {\it cost of a coupling $\pi$} is 
$$\int_{M\times M} c_L(p,q)\; d\pi(p,q).$$
Denote by $C_{L}(\mu,\nu)$ the {\it minimal cost relative to} $c_L$ of couplings between $\mu$ and $\nu$, i.e.
$$C_{L}(\mu,\nu):=\inf\left\{\left.\int c_{L}d\pi\right|\pi\in\Pi(\mu,\nu)\right\}\in \R\cup\{\infty\}.$$
A coupling $\pi$ of two probability measures $\mu$ and $\nu$ is {\it optimal} if the cost of $\pi$ is minimal, i.e.
$$\int c_L d\pi= C_L(\mu,\nu).$$

Denote by $\mathcal{P}(M)$ the set of Borel probability measures on $M$ and set
$$\mathcal{P}_\tau(M):=\{\mu\in \mathcal{P}(M)|\; \tau\in L^1(\mu)\}$$
for a splitting $\tau\colon M\to\R$.

\begin{prop}\label{P1}
Let $\mu,\nu\in \mathcal{P}_\tau(M)$. Then there exists an optimal coupling $\pi$ of $\mu$ 
and $\nu$.
\end{prop}

\begin{proof}
The statement is a direct consequence of \cite[Theorem 4.1]{villani}. One thus has to check the assumption. Manifolds are Polish spaces and $c_L$ is lower 
semicontinuous. For the other assumptions one has to find two upper semicontinuous functions $a,b\colon M\to \R\cup\{-\infty\}$ 
with $a\in L^1(\mu)$, $b\in L^1(\nu)$ and $a(p)+b(q)\le c_L(p,q)$.

The inequality $-d\tau(v)\le L(v)$ for all $v\in TM$ 
implies that $\tau(p)-\tau(q)\le c_L(p,q)$. Thus setting $a:=\tau\in L^1(\mu)$ and $b:=-\tau\in L^1(\nu)$ yields the claim.
\end{proof}

The abstract existence result in Proposition \ref{P1} immediately raises the question: Under what assumptions does a coupling with finite cost exist? 
The simplest case is 
that of two Dirac measures $\mu=\delta_p$ and $\nu=\delta_q$. A coupling of $\delta_p$ and $\delta_q$ with finite cost exists if and only if $(p,q)\in J^+$. In turn 
$(p,q)\in J^+$ if and only if $\delta_q(J^+(A))\ge \delta_p(A)$ and $\delta_p(J^-(B))\ge \delta_q(B)$ for all measurable $A,B\subset M$. The necessity of 
the condition was noticed in \cite{berpue} for relativistic cost functions and general measures.

The problem can be formulated in a more abstract setting though. Let $(\mathcal{X},d_\mathcal{X})$ and $(\mathcal{Y},d_\mathcal{Y})$ be locally compact 
Polish spaces.
Denote by $\pi_\mathcal{X}\colon \mathcal{X}\times \mathcal{Y}\to \mathcal{X}$ and $\pi_\mathcal{Y}\colon \mathcal{X}\times \mathcal{Y}\to \mathcal{Y}$ 
the canonical projections. For $\mathscr{J}\subseteq \mathcal{X}\times \mathcal{Y}$, $A\subseteq \mathcal{X}$ and $B\subseteq \mathcal{Y}$ define
$$\mathscr{J}^+(A):= \pi_\mathcal{Y}(\pi_\mathcal{X}^{-1}(A)\cap \mathscr{J})\subset \mathcal{Y}\text{ and }\mathscr{J}^-(B):= \pi_\mathcal{X}(\pi_\mathcal{Y}^{-1}(B)\cap \mathscr{J})\subset \mathcal{X}.$$

\begin{definition}\label{defcaurel}
Two probability measures $\mu\in \mathcal{P}(\mathcal{X})$ and $\nu\in \mathcal{P}(\mathcal{Y})$ are {\it $\mathscr{J}$-related} if there exists a coupling $\pi$ 
of $\mu$ and $\nu$ with $\pi(\mathscr{J})=1$.
\end{definition}

For $\mathcal{X}=\mathcal{Y}=M$ and $\mathscr{J}=J^+$ a coupling $\pi$ with $\pi(J^+)=1$ is called a {\it causal coupling}. Further for
two probability measures $\mu,\nu\in \mathcal{P}_\tau (M)$ the $J^+$-relation is equivalent to the finiteness of the optimal 
cost, i.e. $|C_{L}(\mu,\nu)|<\infty$. Indeed if $|C_{L}(\mu,\nu)|<\infty$ there exists a coupling $\pi$ of $\mu$ and $\nu$ with $|\int c_Ld\pi| <\infty$, i.e.
$\pi(J^+)=1$. If on the other hand there exists a causal coupling of $\mu$ and $\nu$ then 
$$\left|\int c_Ld\pi'\right|\le \left|\int \tau d\mu\right|+ \left|\int \tau d\nu\right|<\infty$$
for every causal coupling $\pi'$. It follows that $|C_{L}(\mu,\nu)|<\infty$.

\begin{theorem}\label{P2}
Let $(\mathcal{X},d_\mathcal{X})$ and $(\mathcal{Y},d_\mathcal{Y})$ be Polish spaces and $\mathscr{J}\subseteq \mathcal{X}\times \mathcal{Y}$ 
closed. Further let $\mu\in \mathcal{P}(\mathcal{X}),\nu\in \mathcal{P}(\mathcal{Y})$. Then the following are equivalent:
\begin{itemize}
\item[(1)] $\mu$ and $\nu$ are $\mathscr{J}$-related.
\item[(2)] $\nu(\mathscr{J}^+(A))\ge \mu(A)$ and $\mu(\mathscr{J}^-(B))\ge \nu(B)$ for all measurable $A\subseteq \mathcal{X}$ and $B\subseteq \mathcal{Y}$. 
\end{itemize}
\end{theorem}

Recently \cite{eckmil} has proven a similar statement for causally simple Lorentzian spacetimes. Therein the authors study the existence of causal 
couplings for different causality assumptions via causal function, a relaxed notion of time function.

After addressing the existence problem of optimal couplings attention turns towards the structure of the optimal couplings. Recall that a set $A\subseteq M\times M$ 
is {\it $c_L$-cyclically monotone} if 
$$\sum c_L(p_i,q_i)\le \sum c_L(p_i,q_{\sigma(i)})$$
for all $\{(p_i,q_i)\}_{1\le i\le n} \subseteq A$ and all $\sigma \in S(n)$. 

Define
$$\mathcal{P}_\tau^+(M):=\{(\mu,\nu)|\; \mu,\nu\in \mathcal{P}_\tau(M)\text{ are $J^+$-related}\}.$$

\begin{prop}\label{propcycmon}
Let $(\mu,\nu)\in \mathcal{P}_\tau^+(M)$. 
\begin{itemize}
\item[(1)] One has
$$C_L(\mu,\nu)= \sup\left(\int_{M} \phi(q) d\nu(q)-\int_M \psi(p)d\mu(p)\right)$$
where the supremum is taken over the functions $\psi\in L^1(\mu),\phi\in L^1(\nu)$ with 
$\phi(q)-\psi(p)\le c_L(p,q)$.
\item[(2)] Every optimal coupling $\pi$ is concentrated on a $c_L$-cyclic monotone Borel subset 
of $M\times M$.
\end{itemize}
\end{prop}

\begin{proof}
Consider the modified cost function $c_L'(p,q):=c_L(p,q)+\tau(q)-\tau(p)$. Since $0\le L(v)+d\tau(v)$ 
it follows that $c_L'\ge 0$. Now the claim follows from \cite[Theorem 3.1,3.2]{ap}.
\end{proof}

Denote by $\Gamma$ the set of $\mathcal{A}$-minimizers $\gamma\colon [0,1]\to M$ such that 
$$d\tau(\dot\gamma)\equiv \tau(\gamma(1))-\tau(\gamma(0)).$$ 
Set $\EV\colon \Gamma\times [0,1]\to M$, 
$(\gamma,t)\mapsto \gamma(t)$ and $\ev_t:=\ev(.,t)$. For $(p,q)\in J^+$ consider the subspace 
$$\Gamma_{p\to q}:=\ev_1\nolimits^{-1}(q)\cap \ev_0\nolimits^{-1}(p).$$ 

Recall the definition of an dynamical optimal  coupling from \cite{villani}.

\begin{definition}\label{defdynoptcou}
A {\it dynamical optimal coupling} is a probability measure $\Pi$ on $\Gamma$ such that $\pi:=(\ev_0,\ev_1)_\sharp \Pi$ is an optimal coupling between 
$\mu:=(\ev_0)_\sharp\Pi$ and $\nu:=(\ev_1)_\sharp \Pi$. 
\end{definition}

\begin{prop}\label{dynoptcou}
For every $(\mu,\nu)\in \mathcal{P}_\tau^+(M)$ there exists a dynamical optimal coupling $\Pi$ 
for $\mu$ and $\nu$. 
\end{prop}

Define the map $\tev\colon \Gamma\times [0,1]\to PTM$, $(\gamma,t)\mapsto [\dot\gamma(t)]\in PTM_{\gamma(t)}$ where $PTM$ denotes the 
projective tangent bundle. For the canonical projection $P\colon PTM\to M$ one has $\EV =P\circ \tev$. Denote with $\supp\mu$ the support of the 
measure $\mu$. 

\begin{theorem}\label{intermediateregularity}
Let $(\mu,\nu)\in \mathcal{P}_\tau^+(M)$ with $\supp\mu\cap\supp\nu=\emptyset$. Then every dynamical optimal 
coupling $\Pi$ of $\mu$ and $\nu$ has the following property: The canonical projection $P$ restricted 
to the image of $T:=\tev(\supp\Pi\times ]0,1[)$ is injective. 
Further the inverse $(P|_{T})^{-1}$ is locally H\"older continuous with exponent $1/2$. 
\end{theorem}

\begin{examp}
The following example shows the optimality of the H\"older continuity in Theorem 
\ref{intermediateregularity}. Consider Minkowski space $(\R^3,\langle.,.\rangle_1)$, i.e. 
$\langle.,.\rangle_1=dx^2+dy^2-dz^2$ for the natural coordinates $\{x,y,z\}$ on $\R^3$. Set 
$$\mathcal{C}:=\{v\in T\R^3|\; \langle v,v\rangle_1\le 0, dz(v)\ge 0\}$$ 
with the Lorentz-Finsler metric $L|_\mathcal{C}(v)=-\sqrt{|\langle v,v\rangle_1|}$. 

Next let $\Phi\colon \R\times (-\e,\e)\to \R^3$ be the map $(x,\phi)\mapsto (x+\cos\phi,\sin\phi,1)$
for $0<\e<\pi/2$. $\Phi$ is an embedding and for $x$ fixed the curve $\phi\mapsto \Phi(x,\phi)$ parameterizes 
$\partial J^+(x,0,0)\cap \{z=1\}$ near $(x+1,0,1)$. For $\phi\in (-\e,\e)$ the $\Phi$-preimage  
of 
$$\{z=1\}\cap \partial J^+((x-\cos\phi,-\sin\phi, -1))$$ 
near $\Phi(x,\phi)=(x+\cos\phi,\sin\phi,1)$ is described by a smooth function 
$$j_{(x,\phi)}\colon (-\e,\e)\to\R$$ 
with $j_{(x,\phi)}(\phi)=x$, $j_{(x,\phi)}'(\phi)=0$ and $j_{(x,\phi)}''(\phi)>0$. Choose a constant 
$j_{(x,\phi)}''(\phi)<C<\infty$ and consider the function $w\colon [0,\e)\to \R$, $w(\phi)=C\cdot 
\phi^2$.

By diminishing $\e$ if necessary the fact that
$$\{x\le w(\phi)\}\times (-\e,\e)\subset \Phi^{-1}(J^+(w(\phi)-\cos\phi,-\sin\phi,-1))$$
and the choice of $C\in\R$ imply that
$$\Phi^{-1}(J^+(w(\phi)-\cos\phi,-\sin\phi,-1))\cap\text{ graph}(w)\subset (-\infty,w(\phi)]\times [0,\phi]$$ 
for every $\phi\in[0,\e)$.
By restricting $\e$ further one can in fact assume that 
\begin{equation}\label{ex1}
\Phi^{-1}(J^+(w(\phi)-\cos\phi,-\sin\phi,-1))\cap\text{ graph}(w)=\text{ graph}(w|_{[0,\phi]}).
\end{equation}

Now consider the $1$-dimensional Lebesgue measure $\mu$ on the interval $\text{im}(w)=[0,C\cdot\e^2)$
normalized to $1$. Define two maps $\psi_{0,1}\colon [0,C\cdot \e^2)\to \R^3$ by setting 
$$\psi_0(x)=(x-\cos w^{-1}(x),-\sin w^{-1}(x),-1)$$
and
$$\psi_1(x)=(x+\cos w^{-1}(x),\sin w^{-1}(x),1).$$

Denote with
$\mu_0:=(\psi_0)_\sharp\mu$ and $\mu_1:=(\psi_1)_\sharp\mu$. Since 
\begin{equation}\label{ex2}
\psi_1(x)=\psi_0(x)+2(\cos w^{-1}(x),\sin w^{-1}(x),1)
\end{equation}
$\mu_0$ and $\mu_1$ are $J^+$-related. Due to \eqref{ex1} one knows that $(\psi_0)_\sharp \mu|_{[0,x]}$ is coupled by any causal coupling to 
$(\psi_1)_\sharp \mu|_{[0,x]}$ for all $x\in [0,C\cdot \e^2)$. Therefore up to changes on a neglectable set the only possible causal
coupling is induced by \eqref{ex2}. Thus every dynamical coupling $\Pi$ is concentrated on the 
curves 
$$\gamma_x\colon t\mapsto \psi_0(x) +2t(\cos w^{-1}(x),\sin w^{-1}(x),1),$$ 
$x\in [0,C\cdot \e^2)$ and their monotone reparameterizations.

Then the evaluations are 
$$\ev\left(\gamma_x,\frac{1}{2}\right)=\left(x,0,0\right)\text{ and }\tev\left(\gamma_x,\frac{1}{2}\right)
=[(\cos w^{-1}(x),\sin w^{-1}(x),1)].$$ 
The map $(P|_T)^{-1}$ is given by $(x,0,0)\mapsto [(\cos w^{-1}(x),\sin w^{-1}(x),1)]$ and is therefore 
only $\frac{1}{2}$-H\"older.
\end{examp}

The map $(P|_T)^{-1}$ in Theorem \ref{intermediateregularity} is Lipschitz for $m=2$, i.e. if $M$ is a
surface. This is a well known fact for positive definite Lagrangians relying on the fact that 
trajectories (1) solve a differential equation with smooth coefficients and (2) have codimension 
$1$ in a surface. These facts carry over readily to this case.

\begin{theorem}\label{lipreg}
Let $(\mu,\nu)\in \mathcal{P}_\tau^+(M)$ with disjoint supports. Further let $K$ be a compact subset of $\inte\mathcal{C}$, the open interior of $\mathcal{C}$. 
Then the canonical projection $P$ restricted to the image of $\tev(\supp\Pi\times ]0,1[)\cap K$ is injective and its inverse is Lipschitz for every dynamical optimal 
coupling $\Pi$.
\end{theorem}

A set $X\subset M$ is {\it ($\mathcal{C}$-)achronal} if every timelike curve $\eta \colon I\to M$ intersects $X$ at most once. Using a splitting one sees that 
$X$ can be written as the graph of a function $f_X$ over a subset of $N$. With the same proof as for \cite[Proposition 14.25]{oneill}, one sees that $f_X$ is 
locally Lipschitz with respect to the metric induced by $h$. Now one can use a Lipschitz-continuous extension of $f_X$ to $N$ to say that $X$ is the subset of a 
locally Lipschitz hypersurface.

A locally Lipschitz hypersurface $X$ has a tangent  space almost everywhere and with the induced Riemannian metric defines a Lebesgue measure 
$\mathcal{L}_X$ on $X$. A measure concentrated on $X$ is {\it absolutely continuous with respect to the Lebesgue measure} if it is absolutely continuous 
with respect to $\mathcal{L}_X$. Note that this definition is independent of the chosen Riemannian metric since any pair of Lebesgue measures induced 
by Riemannian metrics are absolutely continuous with respect to each other.

Call a hypersurface $Y$ {\it locally uniformly spacelike} if for one (hence every) splitting there exists a locally Lipschitz continuous function $f_Y\colon N\to \R$ 
with $Y$ being the graph of $f_Y$ and for all compact $K\subseteq M$ there exists $\varepsilon >0$, such that the Hausdorff distance between 
$TY_y\cap T^1 M$ and $\mathcal{C}^1:=\mathcal{C}\cap T^1M$ is bounded below by $\varepsilon$ for all $y\in K\cap \Gamma$ such that $TY_y$ exists. 
$T^1M$ denotes the unit tangent bundle of $h$. With these notions the following generalization of in \cite[Theorem 4.3]{berpue} can be given.

\begin{theorem}\label{Thmmonge}
Let $(\mu,\nu)\in\mathcal{P}_\tau^+(M)$. Assume that $\mu$ and $\nu$ are concentrated 
on a locally uniformly spacelike hypersurface $A$ and an achronal set $B$, respectively. 
Further assume that $\mu$ is absolutely continuous with respect to the Lebesgue measure 
on $A$. Then there exists a unique optimal coupling $\pi$ and a Borel map $F\colon M\to M$ 
such that $\pi=(\id,F)_\sharp \mu$. 
\end{theorem}

Uniqueness fails if both $A$ and $B$ are allowed to be achronal only. Consider for example subsets $A, B\subset \partial J^-(p)$ in Minkowski space 
for some $p\in \R^{m}$. For suitable choices of $A$ and $B$ not every optimal coupling is 
supported on a graph. More precisely every causal coupling has vanishing cost, but not every causal coupling is supported on a graph.

Existence fails if $B$ is not assumed to be achronal. An example is given by $\mu$ defined as the $1$-dimensional Lebesgue measure on 
$[0,1]\times\{0\}$ in the $2$-dimensional Minkowski space and $\nu$ a nontrivial superposition of the $1$-dimensional Lebesgue measures on 
$[1,2]\times \{1\}$ and $[2,3]\times\{2\}$. The only possible causal coupling and therefore optimal one splits every point in $[0,1]\times\{0\}$ into two parts 
with weights depending on the superposition. Since the superposition is nontrivial the coupling cannot be induced by a graph.

\begin{theorem}\label{Thmmonge2}
Let $(\mu,\nu)\in \mathcal{P}_\tau^+(M)$. Assume that $\mu$ is absolutely continuous with respect 
to the Lebesgue measure on $M$ and $\nu$ is concentrated on an achronal set $B$. Then there exists a 
unique optimal coupling $\pi$ and a Borel map $F\colon M\to M$ such that $\pi=(\id,F)_\sharp \mu$.
\end{theorem}

Theorem \ref{Thmmonge} corresponds to the classical Monge problem which from the spacetime perspective deals with the problem of coupling two measures
concentrated on different level sets of a splitting $\tau$ (hence time function) and $\mu$ being absolutely continuous with respect to the Lebesgue measure on that level set.
Theorem \ref{Thmmonge2} on the other hand is a version where the initial measure is distributed in space and time, i.e. from the classical point of view 
a family of measures. 

\begin{remark}
In the spirit of the present approach all results are formulated with as little reference to the 
splitting $\tau$ as possible. Note that $\tau$ enters the assumptions of the main results only 
through an integrability condition, i.e. ``$\tau\in L^1(\mu)\cap L^1(\nu)$''. This is automatically 
satisfied for for compactly supported measures. I.e. in this special case all results are indeed 
independent of the splitting. 
\end{remark}

\section{The proofs}\label{proofs}

\subsection{Causal structures}\label{s3.1}

The existence of causal structures is implicitly stated in \cite[page 1534]{minguzzi141} and 
\cite[page 583]{minguzzi142}. The argument is standard material and known for Lorentzian metrics. For 
completeness it is briefly outlined here. 

Let $\mathbb{L}\colon TM\to \R$ be a continuous function positive homogenous of 
degree $2$ and smooth on $TM\setminus T^0M$ such that the second fiber derivative is nondegenerate with index $m-1$. By 
\cite[Proposition 2]{minguzzi141} the number $k$ of connected components of $TM_p\cap \{\mathbb{L}>0\}$ is 
independent of $p\in M$. Thus every point $p\in M$ has a neighborhood $U$ such that 
the fiber bundle $\pi_{TM}^{-1}(U)\cap \{\mathbb{L}>0\}\to U$ is isomorphic to $\sqcup_{i=1}^k 
U\times C_i\to U$ where $C_i$ denotes the forward time cone in the Minkowski $m$-space $\R^m_1$.
Let $\{U_l\}_{l\in \N}$ be a locally finite open covering of $M$ such that 
$$\pi_{TM}^{-1}(U_l)\cap \{\mathbb{L}>0\}\cong \sqcup_{i=1}^k U_l\times C_i.$$
Take the disjoint union 
$$\mathscr{M}:=\sqcup_l (U_l \times \{1,\ldots ,k\})$$
and define ``$\sim$'' to be the equivalence relation generated by $(p,r)\sim (q,s)$ if $p=q$ and $\{p\}\times C_r$ and 
$\{q\}\times C_s$ are mapped to the same connected component of $\{\mathbb{L}>0\}\cap TM_p$ by their respective 
trivializations. Now one shows that the set 
$$M^\mathbb{L}:=\mathscr{M}/\sim$$ 
is a smooth manifold and the map $\pi^\mathbb{L}\colon M^\mathbb{L}\to M$, $[(p,r)]\mapsto p$
is a finite covering. See \cite[Chapter 7]{oneill} for the case of Lorentzian manifolds. 

Consider the pullback $\overline{\mathbb{L}}:=(\pi^\mathbb{L})^*\mathbb{L}$. Then for
every $[(p,r)]\in M^\mathbb{L}$ and every connected component of $TM^\mathbb{L}_{[(p,r)]}\cap 
\{\overline{\mathbb{L}}>0\}$ there exists a vector field $X\in \Gamma(TM^\mathbb{L})$ with 
$X_{[(q,s)]}\in TM^\mathbb{L}_{[(q,s)]}\cap \{\overline{\mathbb{L}}>0\}$ for all $[(q,s)]\in 
M^\mathbb{L}$, i.e. it belongs to a causal structure. This causal structure is unique since the 
components of $TM^\mathbb{L}_{[(p,r)]}\cap \{\overline{\mathbb{L}}>0\}$ are strictly convex.

\subsection{Proof of Theorem \ref{P2}}

(1)$\Rightarrow$(2):  Let $\pi \in \mathcal{P}(\mathcal{X}\times \mathcal{Y})$ be a coupling of $\mu$ and  $\nu$ 
with $\pi(\mathscr{J})=1$. For any set $B\subseteq \mathcal{Y}$ one has 
$$\pi_{\mathcal{Y}}^{-1}(B)\cap \mathscr{J}\subseteq \pi_{\mathcal{X}}^{-1}(\mathscr{J}^-(B)).$$
Since 
$$\nu(B)=\pi(\pi_{\mathcal{Y}}^{-1}(B))=\pi(\pi_{\mathcal{Y}}^{-1}(B)\cap \mathscr{J})$$ 
and 
$$\mu(\mathscr{J}^-(B))=\pi(\pi_{\mathcal{X}}^{-1}(\mathscr{J}^-(B)))$$ 
for $B\subseteq \mathcal{Y}$ measurable, the claim follows. The other inclusion is analogous.

(2)$\Rightarrow$(1): For this part of the proof one needs two lemmata.

\begin{lemma}\label{L2}
Assume that $\mu$ and $\nu$ satisfy the condition in Theorem \ref{P2} (2). 

If there exists a measurable set $A\subseteq \mathcal{X}$ such that 
$\mu(A)= \nu(\mathscr{J}^+(A))\in (0,1)$ then the pairs 
$$(\mu_A,\nu_A):=\left(\frac{1}{\mu(A)}\mu|_{A}, \frac{1}{\mu(A)}\nu|_{\mathscr{J}^+(A)}\right)$$
and
$$(\mu_{A^c},\nu_{A^c}):=\left(\frac{1}{\mu(A^c)}\mu|_{A^c}, \frac{1}{\mu(A^c)}\nu|_{\mathscr{J}^+(A)^c}\right)$$
satisfy the condition in Theorem \ref{P2} (2). 

If $\nu(B)= \mu(\mathscr{J}^-(B))\in (0,1)$ for a measurable set $B\subseteq \mathcal{Y}$ the pairs 
$$(\mu_B,\nu_B):=\left(\frac{1}{\nu(B)}\mu|_{\mathscr{J}^-(B)}, \frac{1}{\nu(B)}\nu|_{B}\right)$$
and 
$$(\mu_{B^c},\nu_{B^c}):=\left(\frac{1}{\nu(B^c)}\mu|_{\mathscr{J}^-(B)^c}, \frac{1}{\nu(B^c)}\nu|_{B^c}\right)$$ 
satisfy the condition in Theorem \ref{P2} (2).
\end{lemma}

\begin{proof}
It suffices to consider the first case. The second case follows by exchange of $\mathcal{X}$ and $\mathcal{Y}$. 
So assume $\mu(A)= \nu(\mathscr{J}^+(A))\in (0,1)$ for some measurable set $A\subseteq \mathcal{X}$. First note that all four measures $\mu_A$, $\nu_A$, 
$\mu_{A^c}$ and $\nu_{A^c}$ are well defined probability measures by the assumption. One has 
\begin{align*}
\mu_A(B)=\frac{1}{\mu(A)} \mu(B\cap A)&\le \frac{1}{\mu(A)} \nu(\mathscr{J}^+(B\cap A))\\
&= \frac{1}{\mu(A)} \nu(\mathscr{J}^+(B)\cap \mathscr{J}^+(A))= \nu_A(\mathscr{J}^+(B))
\end{align*}
which shows $\mu_A(B)\le \nu_A(\mathscr{J}^+(B))$. 

Next note that $\mu(A^c)=\nu(\mathscr{J}^+(A)^c)$. Assume that there exists a measurable set $C\subseteq \mathcal{X}$ with $\nu_{A^c}(\mathscr{J}^+(C))<\mu_{A^c}(C)$, i.e.
$$\nu(\mathscr{J}^+(C)\cap \mathscr{J}^+(A)^c)=\nu|_{\mathscr{J}^+(A)^c}(\mathscr{J}^+(C)) < \mu|_{A^c}(C)=\mu(C\cap A^c).$$
Then a contradiction follows from 
\begin{align*}
\mu(C\cup A)&=\mu(C\cap A^c)+\mu(A)\\
&>\nu(\mathscr{J}^+(C)\cap \mathscr{J}^+(A)^c)+\nu(\mathscr{J}^+(A))=\nu(\mathscr{J}^+(C\cup A))
\end{align*}
since $\mathscr{J}^+(C)\cup \mathscr{J}^+(A)=\mathscr{J}^+(C\cup A)$. Therefore one has 
$$\mu_{A^c}(C)\le \nu_{A^c}(\mathscr{J}^+(C))$$ 
for all measurable $C\subseteq \mathcal{X}$. This shows the first set of inequalities. 

It remains to show $\mu_A(\mathscr{J}^-(D))\ge \nu_A(D)$ and $\mu_{A^c}(\mathscr{J}^-(D))\ge \nu_{A^c}(D)$ for $D\subseteq \mathcal{Y}$ measurable. 
If $\mu_A(\mathscr{J}^-(D))< \nu_A(D)$ one has
$$\mu_A(\mathscr{J}^-(D)^c)= 1-\mu_A(\mathscr{J}^-(D)) >1- \nu_A(D)\ge \nu_A(\mathscr{J}^+(\mathscr{J}^-(D)^c))$$
since $\mathscr{J}^+(\mathscr{J}^-(D)^c)$ and $D$ are disjoint. This contradicts the first part. The inequality $\mu_{A^c}(\mathscr{J}^-(D))\ge \nu_{A^c}(D)$ follows analogously.
\end{proof}

\begin{lemma}\label{L3}
Let $n\in \N$. Consider the product $\{1,\ldots,n\}\times \{1,\ldots,n\}$ with the canonical projections $\pi_{1},\pi_2$ onto the first and second factor, 
respectively. Let $\mathscr{K}\subseteq \{1,\ldots,n\}\times \{1,\ldots,n\}$ have the property that 
\begin{equation}
\sharp \pi_1(\pi_2^{-1}(A)\cap \mathscr{K})\ge \sharp A \text{ and } \sharp \pi_2(\pi_1^{-1}(A)\cap \mathscr{K})\ge \sharp A
\end{equation}
for all $A\subseteq \{1,\ldots,n\}$.
Then $\mathscr{K}$ contains the graph of a permutation $\sigma \in S(n)$.
\end{lemma}

\begin{proof}
The proof is carried out by induction over $n$. If $n=1$ the claim is trivial since $\mathscr{K}= \{1\}\times\{1\}$. 

Now assume that the claim has been shown for numbers less than $n$. First assume that 
$$\sharp \pi_1(\pi_2^{-1}(A)\cap \mathscr{K})> \sharp A\text{ and }\sharp \pi_2(\pi_1^{-1}(A)\cap \mathscr{K})> \sharp A$$ 
for all nonempty proper subsets $A$. Choose $1\le j\le n$ with $(n,j)\in \mathscr{K}$. By renumbering one can 	assume $j=n$. Now consider 
$$\mathscr{I}:= \mathscr{K}\cap  \{1,\ldots,n-1\}\times \{1,\ldots,n-1\}.$$ 
Since 
$$ \sharp{\pi}_1({\pi}_2^{-1}(A)\cap \mathscr{I})\ge  \sharp \pi_1(\pi_2^{-1}(A)\cap \mathscr{K})-1\ge \sharp A$$
and vice versa for all $A\subseteq \{1,\ldots,n-1\}$ one obtains from the induction hypothesis a permutation $o\in S(n-1)$ whose graph is contained
in $\mathscr{I}$. $o$ extends to a permutation $\sigma \in S(n)$ whose graph is a subset of $\mathscr{K}$ by setting $\sigma(n):= n$ and 
$\sigma|_{\{1,\ldots, n-1\}}\equiv o$. 

If there exist a nonempty proper subset $A$ of $\{1,\ldots ,n\}$ with $\sharp \pi_1(\pi_2^{-1}(A)\cap \mathscr{K})= \sharp A$ or 
$\sharp \pi_2(\pi_1^{-1}(A)\cap \mathscr{K})= \sharp A$ one reduces the problem to constructing two separate permutations on 
$A$ and $A^c$. Thus again the induction hypothesis gives separate permutations on $A$ and $A^c$ which together form a permutation $\sigma$ 
whose graph is contained in $\mathscr{K}$. 

One only needs to consider the case $\sharp \pi_1(\pi_2^{-1}(A)\cap \mathscr{K})= \sharp A$. The other case follows by exchanging the order. Further 
by renumbering one can assume that $A= \pi_1(\pi_2^{-1}(A)\cap \mathscr{K})$. Set $\mathscr{K}_A:= \mathscr{K}\cap  A\times A$,  
$\mathscr{K}_{A^c}:= \mathscr{K}\cap  A^c\times A^c$. The goal is to show that $\mathscr{K}_A$ and $\mathscr{K}_{A^c}$ satisfy the assumptions of the lemma. 

It is clear that 
$$\sharp \pi_{1}(\pi_{2}^{-1}(B)\cap \mathscr{K}_A)= \sharp \pi_{1}(\pi_{2}^{-1}(B)\cap \mathscr{K})\ge  \sharp B$$ 
for all $B\subseteq A$ since $\pi_1(\pi_2^{-1}(B)\cap \mathscr{K})\subseteq A$. 
If however there exists $C\subseteq A^c$ with $\sharp \pi_{1}(\pi_{2}^{-1}(C)\cap \mathscr{K}_{A^c})< \sharp C$ then 
$\sharp \pi_1(\pi_2^{-1}(A\cup C)\cap \mathscr{K})< \sharp (A\cup C)$ which contradicts the initial assumption.

Assume now that there exists a set $D\subseteq A$ with $\sharp \pi_{2}(\pi_{1}^{-1}(D)\cap \mathscr{K}_A)< \sharp D$. 
Set $E:= A\setminus \pi_{2}(\pi_{1}^{-1}(D)\cap \mathscr{K}_A)$. Then $D$ and $\pi_{1}(\pi_{2}^{-1}(E)\cap \mathscr{K}_A)$ are disjoint. 
This can be seen as follows. If $i\in \pi_{1}(\pi_{2}^{-1}(E)\cap \mathscr{K}_A)$ then there exists $j\in E$ such that $(i,j)\in \mathscr{K}_A$. 
If $i\in D$ then for all $(i,j)\in \mathscr{K}_A$ one has $j\in \pi_{2}(\pi_{1}^{-1}(D)\cap \mathscr{K}_A)$. Thus the sets are disjoint. It follows that
$$\sharp E=\sharp A-\sharp \pi_{2}(\pi_{1}^{-1}(D)\cap \mathscr{K}_A)>\sharp A-\sharp D\ge \sharp \pi_{1}(\pi_{2}^{-1}(E)\cap \mathscr{K}_A)$$
which clearly contradicts the first part  of the argument. Now the same argument applies to subsets of $A^c$.
\end{proof}

Assume first that there exists $n\in \N$ such that 
$$\mu=\frac{1}{n}\sum_{i=1}^n \delta_{x_i}\text{ and }\nu=\frac{1}{n}\sum_{j=1}^n \delta_{y_j}.$$
Identify $\{x_1,\ldots,x_n\}$ and $\{y_1,\ldots,y_n\}$ with $\{1,\ldots,n\}$. Define the set 
$$\mathscr{K}:=\{(i,j)|(x_i,y_j)\in \mathscr{J}\}\subseteq \{1,\ldots,n\}\times \{1,\ldots,n\}.$$
Denote by $\pi_{1}$ and $\pi_{2}$ the canonical projections from $\{1,\ldots,n\}\times \{1,\ldots,n\}$ onto the first and second factor, 
respectively. Since $\mu$ and $\nu$ are counting measures, the assumptions become 
\begin{equation*}
\sharp \pi_1(\pi_2^{-1}(A)\cap \mathscr{K})\ge \sharp A \text{ and } \sharp \pi_2(\pi_1^{-1}(A)\cap \mathscr{K})\ge \sharp A
\end{equation*}
for all $A\subseteq \{1,\ldots,n\}$.
Lemma \ref{L3} now gives a permutation $\sigma$ whose graph is contained in $\mathscr{K}$. Reversing the identifications one obtains a bijective map
$$\sigma'\colon \{x_1,\ldots ,x_n\}\to \{y_1,\ldots ,y_n\}$$ 
with $(x_i,\sigma'(x_i))\in \mathscr{J}$ for all $i$. Since $\mu$ and $\nu$ are counting measures $(\id,\sigma')_\sharp \mu$ is the desired coupling.

 The general case follows from this special case by an approximation argument.
Choose sequences of locally finite, disjoint and measurable coverings of $\supp \mu$ and $\supp \nu$, respectively.
Then one can approximate both measures in the weak-$*$ topology by finite measures whose support is contained in a given neighborhood of the 
supports of $\mu$ and $\nu$. Consider $\mathscr{J}_\e:=\overline{B_\varepsilon(\mathscr{J})}$, the closure of the $\e$-neighborhood of 
$\mathscr{J}$ with respect to the metric on $\mathcal{X}\times \mathcal{Y}$, for $\varepsilon >0$. Then every pair of 
finite measures $\mu'=\sum b_i\delta_{x_i}$ and $\nu'=\sum c_j \delta_{y_j}$ approximating $\mu$ and $\nu$ sufficiently well, satisfies 
the assumptions in Theorem \ref{P2}(2) for $\mathscr{J}_\e$ instead of $\mathscr{J}$.
In order to apply the special case it would suffice to have $b_i, c_j\in \Q$. Simply approximating the weights $b_i$ and $c_j$ by rational 
numbers and retaining the assumptions of Lemma \ref{L3} will in general only work if 
\begin{equation}\label{EQ1}
\nu'(\mathscr{J}_\e^+(A))>\mu'(A)\text{ and }\mu'(\mathscr{J}_\e^-(B))>\nu'(B)
\end{equation}
for all measurable $A\subset \mathcal{X}$ and $B\subset \mathcal{Y}$. With Lemma \ref{L2} one can split $\mu'$ and $\nu'$ into submeasures 
until \eqref{EQ1} is satisfied and proceed with the submeasures. Since $\mu'$ and $\nu'$ have finite supports this division process terminates after 
finitely many steps. For $\mu'$ or $\nu'$ supported in a single point it is obvious how to build a coupling in $\mathscr{J}_\e$.

If \eqref{EQ1} is satisfied the weights can be approximated by rational numbers such that \eqref{EQ1} still holds for the perturbed measures.
Then by the special case there exists a coupling supported in $\overline{B_\e(\mathscr{J})}$.
By construction the approximations of $\mu$ and $\nu$ form precompact sets in the weak-$*$ topology. This implies that the set of couplings
is precompact in the weak-$*$ topology as well, see \cite[Chapter 4]{villani}. The claim follows when passing to the limit using that
$\mathscr{J}\cap \supp \mu \times \supp \nu$ is closed.

\subsection{Dynamical Optimal Coupling}

For the splitting $\tau\colon M\to \R$ choose a smooth vector field $X_\tau$ on $M$ with $d\tau(X_\tau)\equiv 1$. Then $X_\tau$ is considered to be a vector field on 
$\R\times N$. Define a Lagrange function
$$L_\tau \colon \R\times TN\to \R\cup \{\infty\},\; L_\tau(t,v):=L(X_\tau(t,\pi_{TN}(v))+v).$$ 
Denote by $\mathcal{D}_\tau\subseteq \R\times TN$ the domain of $L_\tau$. $L_\tau$ is continuous on $\mathcal{D}_\tau$ and smooth 
on $\inte \mathcal{D}_\tau$, the interior of $\mathcal{D}_\tau$. Note that $L_\tau |_{\inte \mathcal{D}_\tau}<0$. For $(t,x)\in \R\times N$ set
$\mathcal{D}_{(t,x)}:=\mathcal{D}_\tau\cap (\{t\}\times TN_x)$. The point $(t,v)\in \mathcal{D}_{(t,x)}$ is identified with the vector
$X_\tau(t,x)+v\in \mathcal{C}$.

Denote with $\partial_v^2 L_\tau$ the {\it second fiber derivative of $L_\tau$}, i.e. 
$$(\partial_v^2L_\tau)_{(t,v)}(w,z):=\left.\frac{d^2}{drds}\right|_{r=s=0}L_\tau(t,v+rw+sz).$$

\begin{lemma}\label{strictconvexity}
\begin{itemize}
\item[(i)] $\mathcal{D}_{(t,x)}$ is a compact strictly convex domain with smooth boundary for all $(t,x)\in\R\times N$. 
\item[(ii)] For all $K\subseteq \R\times N$ compact there exists $\delta>0$ such that the second fiber derivative satisfies
$$(\partial_v^2 L_\tau)_{(t,v)}\ge \frac{\delta}{|L_\tau(t,v)|}\cdot \id$$ 
for all $(t,x)\in K$ and $v\in \inte\mathcal{D}_{(t,x)}$. 
\end{itemize}
\end{lemma}

\begin{proof}
(i) Denote with $\mathcal{C}_{(t,x)}^*$ the dual cone of $\mathcal{C}_{(t,x)}=\mathcal{C}_p$ via the identification $(t,x)\cong p$. Then $d\tau_{(t,x)} \in \inte \mathcal{C}_{(t,x)}^*$
since $\tau$ is a Lyapunov function for $\mathcal{C}$. This implies that $\mathcal{D}_{(t,x)}$ is compact since $X_\tau +v\in \mathcal{C}$ yields 
$$1=d\tau(X_\tau +v)\ge |v|-|X_\tau|$$
which bounds the norm of $v$. It is further smooth since 
$\partial \mathcal{C}_{(t,x)}$ is smooth away from the zero section and $\ker d\tau\cap \mathcal{C}=\{0\}$. Finally the strict 
convexity follows from the fact that at points in $\partial\mathcal{C}\setminus T^0M$ the bilinear form $\partial_v^2 \mathbb{L}|_{T\partial\mathcal{C}\times T\partial\mathcal{C}}$ 
is semidefinite with kernel equal to the radial direction, i.e. definite on any hyperplane transversal to the radial direction. Here the radial direction at $v\in TM$ is 
$\frac{d}{ds}|_{s=0}(1+s)v\in T(TM_p)_v\cong TM_x$ and $\partial_v^2 \mathbb{L}$ is defined analogous to $\partial_v^2L_\tau$.

(ii) Recall the formula for the second derivative of $L$ in the fiber direction
$$\partial^2_v L=\frac{1}{2\sqrt{\mathbb{L}}}\left(\frac{1}{2}\frac{\partial_v\mathbb{L}\otimes\partial_v\mathbb{L}}{\mathbb{L}}-\partial^2_v\mathbb{L}\right)$$
where $\partial_v\mathbb{L}_v(w):=\left.\frac{d}{ds}\right|_{s=0}\mathbb{L}(v+sw)$. As seen in (i) one has 
$$\partial^2_v \mathbb{L}|_{T\partial\mathcal{D}_{(t,x)}\times T\partial\mathcal{D}_{(t,x)}}<0$$
for all $(t,x)\in \R\times N$. Thus one can choose $n<\infty$ and $\delta_1 >0$ such that 
$$\left.\left(\frac{n}{2}\partial_v\mathbb{L}\otimes\partial_v\mathbb{L}-\partial^2_v \mathbb{L}\right)\right|_{T\mathcal{D}_\tau\times T\mathcal{D}_\tau}>\delta_1\cdot \id$$
on a neighborhood $U$ of $\partial\mathcal{D}_\tau$ in $\R\times TN$ over $K$. This implies the claim on the smaller neighborhood 
$U\cap\{\mathbb{L}<1/n\}$.

For the remaining points outside of $U\cap\{\mathbb{L}<1/n\}$ note that 
$$\frac{1}{2}\frac{\partial_v\mathbb{L}\otimes\partial_v\mathbb{L}}{\mathbb{L}}-\partial^2_v\mathbb{L}\ge 0$$
with kernel equal to the radial direction. Thus one has 
$$\frac{1}{2}\frac{\partial_v\mathbb{L}\otimes\partial_v\mathbb{L}}{\mathbb{L}}-\partial^2_v\mathbb{L}> \delta_2\cdot \id$$
on $\mathcal{D}_\tau\setminus U\cap\{\mathbb{L}<1/n\}$ over $K$ for a $\delta_2>0$.
\end{proof}

Let $V\subset N$ be open with a chart $V\to \R^{m-1}$ of $N$. The induced trivialization of $TV\to T\R^{m-1}\cong \R^{m-1}\times \R^{m-1}$ gives local coordinates 
$(x,v)\in \R^{m-1}\times \R^{m-1}$ on $TN$. The Euler-Lagrange equation of the action functional associated to $L_\tau$ reads in these coordinates:
\begin{equation}\label{ELE1}
\frac{\partial L_\tau}{\partial t}+\frac{\partial L_\tau}{\partial x}-\frac{d}{dt}\left(\frac{\partial L_\tau}{\partial v}\right)=0
\end{equation}
The equation defines an explicit ordinary differential equation of second order since $\frac{\partial^2 L_\tau}{\partial v^2}>0$ at points in $\inte \mathcal{D}_\tau$.
It is standard that the solution to \eqref{ELE1} are of the form $t\mapsto \dot{\eta}(t)$ for some curve $\eta\colon I\to N$, i.e. the solutions are tangent curves in $TN$. 
For $(t,v)\in \inte \mathcal{D}_\tau$ denote with $\eta_{(t,v)}\colon I\to N$ the unique maximal solution to \eqref{ELE1} with $\dot\eta_{(t,v)}(0)=v$.
The solutions define a local flow 
$$\Phi_\tau\colon U_\tau \to \inte\mathcal{D}_\tau,\; (s,(t,v))\mapsto (s+t,\dot{\eta}_{(t,v)}(s))$$ 
where $U_\tau\subset \R\times \inte \mathcal{D}_\tau$ is is an open neighborhood of $\{0\}\times \inte \mathcal{D}_\tau$.

\begin{prop}\label{flowtau}
$\Phi_\tau$ extends to a smooth local flow on an open neighborhood of $\{0\}\times \mathcal{D}_\tau$, i.e. there 
exists an open neighborhood $U$ of $\{0\}\times \mathcal{D}_\tau$ in $\R\times \R\times TN$ and $\tilde\Phi_\tau\colon U\to \R\times TN$ smooth with 
$\tilde\Phi_\tau \equiv \Phi_\tau$ on $U_\tau$. Furthermore $\tilde\Phi_\tau$ is complete on $\mathcal{D}_\tau$ with $\inte\mathcal{D}_\tau$ and $\partial\mathcal{D}_\tau$ 
$\tilde\Phi_\tau$-invariant. The extension of $\Phi_\tau$ to $\mathcal{D}_\tau$ is unique and will be denoted by $\Phi_\tau$ again.
\end{prop}

For a local trivialization $W\times \R^m$ of $TM$ with coordinates $(p,w)\in W\times \R^m$ the Euler-Lagrange equation of the action functional associated to $\mathbb{L}$
\begin{equation}\label{ELE2}
\frac{d}{dt}\left(\frac{\partial\mathbb{L}}{\partial w}\right)-\frac{\partial\mathbb{L}}{\partial p}=0
\end{equation}
defines a local flow outside the zero section since $\frac{\partial^2\mathbb{L}}{\partial w^2}$ is nondegenerate. For $w\in TM\setminus T^0M$ let 
$\gamma_w\colon J\to M$ be the unique maximal solution to \eqref{ELE2} with $\dot\gamma_w(0)=w$. Denote with 
$\Phi^\mathbb{L}\colon\mathbb{U}\subset \R\times TM\to TM,\; (t,w)\mapsto \dot\gamma_w(t)$ the maximal flow defined by \eqref{ELE2} extended to the zero 
section by constant flow lines, see \cite{minguzzi142}. $\Phi^\mathbb{L}$ is smooth outside the zero section.

A causal curve $\eta \colon I\to N$ is a {\it $\Phi_\tau$-trajectory} if $t\mapsto (t,\dot{\eta}(t))$ solves \eqref{ELE1}. A curve $\gamma\colon J\to M$ is a 
{\it $\Phi^\mathbb{L}$-trajectory} if $t\mapsto \dot\gamma(t)$ solves \eqref{ELE2}.

\begin{lemma}\label{phitau}
A curve $\eta\colon I\to N$ with $(t,\dot\eta(t))\in \inte \mathcal{D}_\tau$ for all $t\in I$ is a $\Phi_\tau$-trajectory if and only if its graph $H\colon t\mapsto (t,\eta(t))$ 
is a reparameterization of a $\Phi^\mathbb{L}$-trajectory $\gamma$ with $\dot\gamma\in\inte\mathcal{C}$. Especially the trajectories of $\Phi_\tau$ 
and $\Phi^\mathbb{L}$ are in one-to-one correspondence via reparameterization.
\end{lemma}

\begin{proof}
Fix a local chart of $W \to \R^m$ of $M$ and the induced trivialization of $TW\to T\R^m \cong \R^m\times \R^m$. Denote by 
$(p,w)\in \R^m\times \R^m$. Expanding the Euler-Lagrange equation of $L$ on $\inte \mathcal{C}$ gives
\begin{equation}\label{eqEL}
0= \frac{d}{dt}\left(\frac{\partial L}{\partial w}\right)-\frac{\partial L}{\partial p}=\frac{d}{dt}\left(\frac{1}{2\sqrt{\mathbb{L}}}\right)\frac{\partial \mathbb{L}}
{\partial w}+\frac{1}{2\sqrt{\mathbb{L}}}\left[\frac{d}{dt}\left(\frac{\partial \mathbb{L}}{\partial w}\right) -\frac{\partial \mathbb{L}}{\partial p}\right].
\end{equation}
Since $\mathbb{L}$ is autonomous, $\mathbb{L}$ is preserved along orbits of the local Euler-Lagrange flow $\Phi^\mathbb{L}$ of $\mathbb{L}$. 
This immediately show that $\inte\mathcal{C}$ and $\partial\mathcal{C}\setminus T^0M$ are invariant under $\Phi^\mathbb{L}$. Thus according 
to \eqref{eqEL} an orbit of $\Phi^\mathbb{L}$ in $\inte\mathcal{C}$ solves the Euler-Lagrange equation of $L$. Conversely let $\theta\colon I\to M$ 
solve the Euler-Lagrange equations of $L$. Reparameterizing $\theta$ to a curve $\gamma$ such that $\mathbb{L}$ is constant along $\dot\gamma$ yields 
an orbit of $\Phi^\mathbb{L}$. 

Now one shows that $\eta$ is a $\Phi_\tau$-trajectory if and only if $H$ solves the Euler-Lagrange equation of $L$. 
Let $\eta\colon I\to N$ be a $\Phi_\tau$-trajectory. Consider a smooth variation $\mathcal{H}\colon I\times (-\e,\e)\to \R\times N$ of 
$H$ with fixed endpoints. Since $\mathcal{H}$ is smooth one can assume, by diminishing $\e$ if necessary, that $\partial_t(\tau
\circ\mathcal{H})>0$ everywhere. Thus one can smoothly reparameterize $\mathcal{H}$ to satisfy $\partial_t(\tau
\circ\mathcal{H})=1$, i.e. $\mathcal{H}$ consists of graphs of curves $\eta_s\colon I\to N$ ($s\in (-\e,\e)$).
This shows that any sufficiently small variation of $H$ can be reparameterized to be a variation by graphs. 
The reparameterization does not affect the value of $\mathcal{A}$ on the variation. Note that 
$$\int L_\tau(t,\dot{\eta}_s(t))dt=\int L(\partial_t \mathcal{H}(t,s))dt.$$
Now if the first variation of $\eta$ vanishes the first variation of the graph vanishes as well, i.e. $H$ solves the Euler-Lagrange equations of $L$. 
The converse is obvious, i.e if $H$ solves the Euler-Lagrange equations of $L$, then the first variation of $\eta$ vanishes.

Combining both paragraphs gives the first claim. For the second claim one has to use the positive homogeneity of $\Phi^\mathbb{L}$, i.e. 
$\Phi^\mathbb{L}(\lambda t,w)= \Phi^\mathbb{L}(t,\lambda w)$ for $\lambda>0$. Thus reparameterizing a $\Phi^\mathbb{L}$-trajectory to a curve $\gamma$
with $d\tau(\dot\gamma)\equiv 1$ gives the same curves for initial values $w$ and $\lambda w$ where $\lambda>0$. It remains to note that half lines in 
$\inte\mathcal{C}$ are in one-to-one correspondence with points in $\mathcal{D}_\tau$.
\end{proof}

Recall that $\pi_{TM}\colon TM\to M$ denotes the canonical projection.

\begin{proof}[Proof of Proposition \ref{flowtau}]
Since $\Phi^\mathbb{L}$ is a smooth local flow on $TM\setminus T^0M$ every $v\in \partial \mathcal{C}\setminus T^0M$ admits an $\e(v)>0$ 
and a neighborhood $W$ in $TM\setminus T^0M$ such that $d\tau(\Phi^\mathbb{L}(t,w))>0$ for all $|t|\le \e(v)$ and $w\in W$.

Parameterize the trajectories $t\mapsto 
\pi_{TM}(\Phi^\mathbb{L}(t,w))$ to curves $\gamma_w$ such that 
\begin{equation}\label{eqRPA}
d\tau(\dot\gamma_w)\equiv 1\text{ and }\dot\gamma_w(0)=\frac{w}{d\tau(w)}.
\end{equation}
Since $\Phi^\mathbb{L}(t,\lambda v)=\Phi^\mathbb{L}(\lambda t, v)$ for $\lambda >0$ the curves $\gamma_w$ and $\gamma_{\lambda w}$ coincide for positive $\lambda$. 
The tangent curves $t\mapsto \dot\gamma_w(t)$ define a local flow. This is due to the fact that 
$$\dot\gamma_w(s+t)=\dot\gamma_{\dot\gamma_w(s)}(t)$$
for $|s|,|t|$ sufficiently small. Now since $\Phi^\mathbb{L}$ is autonomous these local definitions of the extensions patch together to give a local flow on a 
neighborhood of $\mathcal{C}\cap\{d\tau=1\}$ in $\{d\tau=1\}$ which preserves $\inte\mathcal{C}\cap\{d\tau=1\}$ and $\partial\mathcal{C}\cap\{d\tau=1\}$. 
Projecting the flow to $\R\times TN$ gives a smooth extension since the projection coincides with $\Phi_\tau$ on $\inte\mathcal{D}_\tau$ by Lemma 
\ref{phitau}. 

It remains to prove the completeness of the extension. But this follows directly from Remark \ref{remarkinextendable} since the $\gamma_w$ are causal for 
$w\in\mathcal{C}$ and $\tau(\gamma_w(t))-\tau(\gamma_w(s))=t-s$.
\end{proof}

The Riemannian metric $h$ induces a Riemannian metric on all higher tangent bundles 
$T^{(k)}M$ where $T^{(k)} M:= T(T^{(k-1)}M)$ and $T^{(1)} M:= TM$. For $a\le b\in \R$ define the $C^k$-topology 
on smooth curves $\gamma\colon [a,b]\to M$ via the induced metrics as
$$\dist\nolimits_k (\gamma,\eta):=\sup\{\dist (\gamma^{(k)}(t),\eta^{(k)}(t))|\;
t\in[a,b] \}.$$

\begin{lemma}\label{L3.6}
For all $(p,q)\in J^+$ the set $\Gamma_{p\to q}$ is nonempty, consists of smooth curves and
is compact in the $C^k$-topology for all $k$.
\end{lemma}

\begin{proof}
By Proposition \ref{P0} there exists an $\mathcal{A}$-minimizer between $p$ and $q$ and every 
$\mathcal{A}$-minimizer solves the Euler-Lagrange equations of $\mathbb{L}$ up to monotone reparameterization. 
Monotonously reparameterizing an $\mathcal{A}$-minimizer between $p$ and $q$ to $\gamma\colon [0,1]\to M$ with $d\tau(\dot\gamma)\equiv 
\tau(q)-\tau(p)$ yields $\gamma\in \Gamma_{p\to q}$, i.e. $\Gamma_{p\to q}$ is nonempty. Every curve in $\Gamma_{p\to q}$ solves the
Euler-Lagrange equations of $L$. 

$\Gamma_{p\to q}$ contains only the constant curve if $p=q$, i.e. in this case $\Gamma_{p\to q}$ is compact in every topology. 
If on the other hand one has $p\neq q$, an $\mathcal{A}$-minimizer $\gamma\in \Gamma_{p\to q}$ induces a 
$\Phi_\tau$-trajectory $\eta\colon [\tau(p),\tau(q)]\to N$ via the graph 
$$H(t)=(t,\eta(t)):=\gamma\left(\frac{t-\tau(p)}{\tau(q)-\tau(p)}\right)$$ 
of $\eta$ by Lemma \ref{phitau}. Identify $p\cong (\tau(p),x)$ and $q\cong (\tau(q),y)$ via the splitting $M\cong \R\times N$. 
The set of $\Phi_\tau$-trajectories $\delta\colon [\tau(p),\tau(q)]\to N$ between $x$ and $y$ is compact in all $C^k$-topologies
on $C^\infty([\tau(p),\tau(q)],N)$ by Proposition \ref{flowtau} since it is part of a smooth flow. Now the compactness of $\Gamma_{p\to q}$ is obvious.
\end{proof}

The following results are analogous to results in \cite[chapter 7]{villani}. 

\begin{prop}
There exists a Borel map $S\colon J^+ \to C^0([0,1],M)$ such that $S(p,q)\in \Gamma_{p\to q}$. 
\end{prop}

\begin{proof}
For every $(p,q)\in J^+$ the set $\Gamma_{p\rightarrow q}$ is nonempty and compact in every $C^k$-topology by the Lemma \ref{L3.6}, 
i.e. nonempty and closed. Further the evaluation map $\ev_0\times \ev_1$
is Lipschitz. This implies that the correspondence (for the definition see \cite[page 4]{alibor})
$$(\ev\nolimits_0\times \ev\nolimits_1)^{-1}\colon J^+\twoheadrightarrow  \Gamma$$
is weakly measurable in the sense of \cite[Definition 18.1]{alibor}. Now \cite[Theorem 8.13]{alibor} 
implies that $(\ev_0\times \ev_1)^{-1}$ has a measurable selection $S$, i.e. $(\ev_0\times \ev_1)
\circ S\equiv \id|_{J^+}$.
\end{proof}

\begin{proof}[Proof of Proposition \ref{dynoptcou}]
Let $(\mu,\nu)\in \mathcal{P}^+_\tau(M)$ and let $\pi$ be an optimal coupling of $\mu$ and $\nu$ for the cost $c_L$. Consider 
$\Pi:=S_\sharp \pi$. Since $(\ev_0,\ev_1)\circ S\equiv \id$, the claim follows from the definition of 
optimal dynamical couplings. 
\end{proof}

\begin{cor}\label{intermcoup}
Let $\Pi$ be a dynamical optimal coupling between $J^+$-related measures $\mu_0$ and $\mu_1$ and $\sigma_{1},\sigma_2\colon \Gamma\to [0,1]$ measurable
functions with $\sigma_1\le \sigma_2$. Then the restriction 
$$\pi_{\sigma_1,\sigma_2}:=(\EV\circ (\id\times \sigma_1),\EV\circ (\id\times \sigma_2))_\sharp\Pi$$ 
is an optimal coupling of $\mu_{\sigma_1}:= (\EV\circ (\id\times \sigma_1))_\sharp\Pi$ and $\mu_{\sigma_2}:= (\EV\circ 
(\id\times \sigma_2))_\sharp\Pi$. If furthermore $(\sigma_1,\sigma_2)\neq(0,1)$ $\Pi$-almost everywhere then $\pi_{\sigma_1,\sigma_2}$ is
the unique optimal coupling of $\mu_{\sigma_1}$ and $\mu_{\sigma_2}$. 
\end{cor}

\begin{proof}
By the triangle inequality for $c_L$ and the parameterization invariance of $\mathcal{A}$ one has
$$C_L(\mu_0,\mu_1)\le C_L(\mu_0,\mu_{\sigma_1})+C_L(\mu_{\sigma_1},\mu_{\sigma_2})+C_L(\mu_{\sigma_2},\mu_1)$$
and
$$\int c_L d\pi_{0,1}=\int c_L d\pi_{0,\sigma_1}+\int c_L d\pi_{\sigma_1,\sigma_2}
+\int c_L d\pi_{\sigma_2,1}.$$
Since $\int c_L d\pi_{0,1}=C_L(\mu_0,\mu_1)$ and $C_L$ is the minimal cost, the three terms on the 
right hand sides must individually coincide. 
More precisely, if one coupling on the right hand side, without loss of generality $\pi_{0,\sigma_1}$, 
is not optimal one can replace this coupling by an optimal coupling $\pi'$ with strictly smaller cost. 
Let $\pi'$ be an optimal coupling of $(\ev_0)_\sharp \Pi$ and $(\ev\circ (\id\times \sigma_1))_\sharp 
\Pi$.  Then one has 
$$\int c_L d\pi_{0,1}>\int c_L d\pi'+\int c_L d\pi_{\sigma_1,\sigma_2}+\int c_L d\pi_{\sigma_2,1}.$$
Gluing the three couplings gives a coupling of $\mu_0$ and $\mu_1$ with strictly smaller cost, a
contradiction.
The second statement follows directly from the triangle inequality for $c_L$, see section 
\ref{results}.
\end{proof}

\begin{cor}\label{cor_interpolation}
Let $(\mu_{0},\mu_1)\in \mathcal{P}_\tau^+(M)$. Further let $\Pi$ be a dynamical optimal coupling of $\mu_0$ and $\mu_1$.
If $\Xi$ is a measure on $\Gamma$, such that $\Xi\le \Pi$ and $\Xi(\Gamma)>0$, set 
$$\Xi':=\frac{\Xi}{\Xi(\Gamma)} \text{ and } \nu_{i}:=(\ev\nolimits_{i})_\sharp\Xi'$$
for $i=0,1$. Then $\Xi'$ is a dynamical optimal coupling between $\nu_{0}$ and $\nu_{1}$.
\end{cor}

\begin{proof}
The assumption $\Xi\le \Pi$ implies that $\Pi-\Xi$ is a measure on $\Gamma$. Set $\pi:=
(\ev_0,\ev_1)_\sharp \Pi$, $\pi^0:=(\ev_0,\ev_1)_\sharp \Xi$ and $\pi^1:=(\ev_0,\ev_1)_\sharp
(\Pi-\Xi)$. Then one has $\pi =\pi^0+\pi^1$ and 
\begin{equation}\label{E11}
\int c_L d\pi=\int c_L d\pi^0+\int c_L d\pi^1.
\end{equation}
Since the left hand side of \eqref{E11} is optimal so must be the terms on the right hand side. Here optimality of $\pi^0$ 
means optimality of the cost $\int c_Ld\pi'$ among all Borel measures $\pi'$ on $M\times M$ with $\pi'(M\times M)
=\pi^0(M\times M)=\Xi(\Gamma)$ and martingales equal to $\Xi(\Gamma)\nu_0$ and $\Xi(\Gamma)\nu_1$. Optimality of $\pi^1$ 
is defined analogously. Now if $\pi^0$ was not optimal one could replace $\pi^0$ by a coupling $\sigma$ of 
$(\ev_0)_\sharp \Xi$ and $(\ev_1)_\sharp \Xi$ with strictly smaller cost. $\sigma+\pi^1$ is a coupling of $\mu_0$ and $\mu_1$ by 
construction, but $\int c_L d\pi >\int c_L d\sigma +\int c_L d\pi^1$, a contradiction. This shows $\Xi'$ to be a dynamical optimal coupling.
\end{proof}

\subsection{Intermediate regularity of dynamical optimal couplings}\label{intermregu}

Recall that $\mathbb{U}\subseteq \R\times TM$ denotes the maximal domain of the Euler-Lagrange flow $\Phi^\mathbb{L}$ of $\mathbb{L}$.
Let $U$ be defined by $\{1\}\times U:=(\{1\}\times TM)\cap \mathbb{U}$. $U$ is a fiberwise star shaped neighborhood of the zero section. 
Define the exponential map $\exp^\mathbb{L}\colon U\to M\times M$ by 
$$\exp^\mathbb{L}(v):= (\pi_{TM}(v),\pi_{TM}\circ \Phi^\mathbb{L}(1,v))$$ 
 
 \begin{prop}\label{Localdiff}
$\exp^\mathbb{L}$ is a $C^1$-diffeomorphism on a neighborhood of $T^0M$ in $TM$ onto its image. Further $\exp^\mathbb{L}$ is smooth 
outside $T^0M$. 
\end{prop}
 
\begin{proof}
Choose local coordinates $V\to \R^m$ on $M$ and consider the induced coordinates $TV\to T\R^m\cong \R^m\times \R^m$ on $TM$. The chart 
$V\to \R^m$ induces coordinates $V\times V\to \R^m\times \R^m$ on $M\times M$ as well.

Abbreviate $\Phi^\mathbb{L}_t:= \Phi^\mathbb{L}(t,.)$.
In order to show continuous differentiability it suffices to show this for $\pi_{TM}\circ \Phi^\mathbb{L}_1$ at the zero section, 
since $\pi_{TM}$ is everywhere smooth and $\Phi^\mathbb{L}_1$ is smooth outside of the zero section. 

(1) For all $(W,Z)\in \R^m\times \R^m$ the directional derivatives 
$$(p,v)\mapsto \partial_{(W,Z)} (\pi_{TM}\circ \Phi^\mathbb{L}_1)_{(p,v)}$$ 
exist on $U$ and 
$$d(\exp^\mathbb{L})_{{0}_p}(W,Z)=(W,W+Z)$$ 
in the above coordinates. The only points to check are the zero section. Let $(W,Z)\in T(TM)_{{0}_p}$ for $p\in V$. Then one has 
\begin{equation}\label{diffquot}
\frac{1}{t}\left(\pi_{TM}\circ \Phi^\mathbb{L}_1(p+tW,tZ)-\pi_{TM}\circ \Phi^\mathbb{L}_1(p,0)\right)=\frac{1}{t}\left(\pi_{TM}\circ \Phi^\mathbb{L}_1(p+tW,tZ)-p\right).
\end{equation}
For $Z=0$ one has 
$$\frac{1}{t}\left(\pi_{TM}\circ \Phi^\mathbb{L}_1(p+tW,0)-p\right)=\frac{1}{t}(p+tW-p)=W.$$ 
For $Z\neq 0$ it follows that the right hand side of \eqref{diffquot}
converges for $t\to 0$ to 
\begin{align*}
& d\pi_{TM}\left(\left.\frac{d}{dt}\right|_{t=0} \Phi^\mathbb{L}_t(p+tW,Z)\right)\\
=&\; d\pi_{TM}\left(\left.\frac{d}{dt}\right|_{t=0} \Phi^\mathbb{L}_0(p+tW,Z)+\left.\frac{d}{dt}\right|_{t=0} \Phi^\mathbb{L}_t(p,Z)\right)\\
=&\; d\pi_{TM}((W,0)+(Z,0))=W+Z
\end{align*}
since $d\pi_{TM}(W,Z)=W$. This proves the claim. 
 
(2) The directional derivatives $(p,v)\mapsto d(\pi_{TM}\circ \Phi^\mathbb{L}_1)_{(p,v)}(W,Z)$ are continuous. This implies that 
$\exp^\mathbb{L}$ is $C^1$ on $U$ by a standard theorem of calculus. The diffeomorphism property follows from the inverse 
function theorem and (1), since 
$$d(\exp^\mathbb{L})_{{0}_p}(W,Z) =(W,W+Z).$$
 
 In order to show that the directional derivatives are continuous let $(p,v)\in TV \cap U$ with $v\neq 0$ and $(W,Z)\in T(TM)_{(p,v)}$. One has 
 $$\partial_{(W,Z)} \left(\pi_{TM}\circ \Phi^\mathbb{L}_1\right)_{(p,v)}= \left.\frac{d}{dt}\right|_{t=0} \pi_{TM}\circ\Phi^\mathbb{L}_1(p+tW,v)+
 \left.\frac{d}{dt}\right|_{t=0} \pi_{TM}\circ\Phi^\mathbb{L}_1(p,v+tZ),$$
 since $\pi_{TM}\circ \Phi^\mathbb{L}_1$ is smooth around $(p,v)$. 
 
 Setting $\e:=\sqrt{h(v,v)}$ one has for the first term
 \begin{align*}
 \left.\frac{d}{dt}\right|_{t=0} \pi_{TM}\circ\Phi^\mathbb{L}_1(p+tW,v)&= \left.\frac{d}{dt}\right|_{t=0} \pi_{TM}\circ\Phi^\mathbb{L}_\e\left(p+tW,\frac{v}{\e}\right)\\
&=d\pi_{TM}(d\Phi^\mathbb{L}_\e)_{(p,v/\e)}(W,0).
 \end{align*}
The last term converges to $W$ uniformly on compact subsets of $M$ for $\e\to 0$, because $v/\e$ is bounded away from the zero section. 
 
For the second term one has 
\begin{align*}
\left.\frac{d}{dt}\right|_{t=0} \pi_{TM}\circ\Phi^\mathbb{L}_1(p,v+tZ)&=  \left.\frac{d}{dt}\right|_{t=0} \pi_{TM}\circ\Phi^\mathbb{L}_\e\left(p,\frac{v+tZ}{\e}\right)\\
&=d\pi_{TM}(d\Phi^\mathbb{L}_\e)_{(p,v/\e)}\left(0,\frac{Z}{\e}\right).
\end{align*}
The last term equals $Y(\e)$ where $Y$ is the unique solution of the Jacobi equation of $\mathbb{L}$ along $\gamma\colon t\mapsto 
\pi_{TM}\circ \Phi^\mathbb{L}_t(p,v/\e)$ with $Y(0)=0$ and $\dot{Y}(0)=\frac{Z}{\e}$. Since $v/\e\neq 0$ one can write 
$$\left(\begin{matrix} Y \\ \dot{Y}\end{matrix}\right) (t) = \exp \left[ \int_0^t A_\gamma(\sigma)d\sigma\right] \left(\begin{matrix} 0 
\\ \frac{Z}{\e}\end{matrix}\right)$$
for a curve $t\to A_\gamma (t)$ of $2m\times 2m$ matrices. Since $v/\e$ is bounded away from the zero section, $A_\gamma$ is 
uniformly bounded for bounded flow parameters $t$. Further since the Jacobi equation is an equation of second order, $A_\gamma(t)$ 
has the form 
$$A_\gamma(t)= \left(\begin{matrix} 0 & E_n \\ B_\gamma (t) & C_\gamma(t) \end{matrix}\right).$$
Thus 
\begin{align*}
\lim_{\e\to 0} Y(\e)&= \lim_{\e\to 0} \frac{1}{\e} d\pi_{TM} \left(\exp \left[ \int_0^\e A_\gamma(\sigma)d\sigma\right] \left(\begin{matrix} 0 \\ Z\end{matrix}\right)\right)\\
&=\lim_{\e\to 0} \frac{1}{\e}(E_n 0+ 0+\e \cdot E_n Z)=Z.
\end{align*}
This shows that the partial derivatives are continuous. 
\end{proof}

Denote by $\mathcal{A}_\tau$ the action of $L_\tau$, i.e. for $\eta\colon [s,t]\to N$ set 
$$\mathcal{A}_\tau(\eta):=\int_s^t L_\tau(\sigma,\dot\eta(\sigma))d\sigma\in \R\cup\{\infty\}.$$
Define the sets $J^\pm ((s,x))$ and and $I^\pm ((s,x))$ via the splitting $M\cong \R\times N$.

\begin{lemma}\label{normaltau}
Every $(s,x)\in \R\times N$ has a neighborhood $V\subseteq \R\times N$ such that for every $(t,y)\in V\cap J^+((s,x))$ the unique $\Phi_\tau$-trajectory 
$\eta\colon [s,t]\to N$ from $x$ to $y$ strictly minimizes $\mathcal{A}_\tau$ among all curves $\alpha\colon [s,t]\to N$ from $x$ to $y$. 
\end{lemma}

\begin{proof}
As usual define $\exp^\mathbb{L}_p:= \pi_{TM}\circ \Phi^\mathbb{L}(1,.)|_{U\cap TM_p}$  at $p\in M$.
According to Proposition \ref{Localdiff} every point $p\in M$ admits a normal neighborhood $V$, i.e. $\exp^\mathbb{L}_p$ is a diffeomorphism from a neighborhood of $0_p$ onto $V$. 
Further according to \cite{minguzzi142} every point $q\in J^+_V(p)$ (recall the definition from section \ref{results}) is connected to $p$ via a unique 
$\Phi^\mathbb{L}$-trajectory $\gamma$ with $\gamma(0)=p$ and $\gamma(1)=q$. $\gamma$ strictly minimizes $\mathcal{A}$ among all causal curves in $V$ 
from $p$ to $q$ up to monotone reparameterizations. Since $\tau(q)-\tau(p)$ bounds the $h$-length of a causal curves (see section \ref{results}) between 
$p$ and $q$ every causal curve between $p$ and $q$ is contained in $V$ given $p$ and $q$ belong to a sufficiently small 
subneighborhood. Thus the $\Phi^\mathbb{L}$-trajectory $\gamma$ is strictly minimal among all causal curves in $M$ from $p$ to $q$ up to monotone 
reparameterizations. Since causal curves are the only curves in $M$ with finite $\mathcal{A}$-action the strict minimality up to monotone reparameterizations 
even holds for all curves in $M$ between $p$ and $q$. 

Since $\mathcal{A}_\tau(\eta)=\mathcal{A}(H)$ for every curve $\eta\colon [s,t]\to N$, where $H\colon [s,t]\to M$ denotes the graph of $\eta$, the local 
minimality follows for the $\Phi_\tau$-trajectories.
Strict minimality follows from the fact that the conditions $\tau\circ H (u)=u$ fix the parameterization of $H\colon [s,t]\to M$ 
uniquely. Define $V\subset \R\times N$ via the identification of $\R\times N\cong M$.
\end{proof}

\begin{remark}\label{localL}
For $(s,x)\in \R\times N$, $(t,y)\in U\cap J^+((s,x))$ as in Lemma \ref{normaltau} and $\eta\colon [s,t]\to N$ the unique 
$\Phi_\tau$-trajectory from $x$ to $y$ one has
$$\mathcal{A}_\tau(\gamma)=L(v),$$
where $v=\left(\exp_{(s,x)}^\mathbb{L}\right)^{-1}(t,y)$. Further denote by $\mathcal{S}_s^t(x,y)$ the minimal action of a 
curve from $x$ to $y$ with respect to $\mathcal{A}_\tau$. Then the previous equality and Lemma \ref{normaltau} imply
$$\mathcal{S}_s^t(x,y)=L\left(v\right)$$
with $v$ as before. Thus $(t,y)\mapsto \mathcal{S}_s^t(x,y)$ is smooth for $(t,y)\in I_V^+((s,x))$ and $V$ as in 
Lemma \ref{normaltau} as follows from Proposition \ref{Localdiff}.
\end{remark}

\begin{prop}\label{mathershortening}
Let $\varepsilon >0$ and $I\times K\subseteq \R\times N$ a compact subset. Then there exist $\delta,\kappa >0$ and $C<\infty$ such that for $a,b,c\in I$ with 
$b-a,c-b\ge \e$ and 
$\Phi_\tau$-trajectories $x_i\colon [a,c]\to N$, $i=1,2$, with $\dist(x_1(b),x_2(b))\le \delta$, $x_i(b)\in K$ and 
$$\dist(\dot{x}_1(b),\dot{x}_2(b))^{2}\ge C\dist(x_1(b),x_2(b))$$ 
there exist $\mathcal{A}_\tau$-minimizers $y_i\colon [a,c]\to N$ with $y_1(a)=x_1(a)$, $y_1(c)=x_2(c)$, $y_2(a)=x_2(a)$, $y_2(c)=x_1(c)$ and 
$$\mathcal{A}_\tau(y_1)+\mathcal{A}_\tau(y_2)-\mathcal{A}_\tau(x_1)-\mathcal{A}_\tau(x_2)\le -\kappa \dist(\dot{x}_1(b),\dot{x}_2(b))^2.$$
\end{prop}

A curve $\eta\colon [a,b]\to N$ is {\it causal} if the graph $H$ is causal in $\R\times N\cong M$.

\begin{lemma}\label{actionestimate}
Let $I\times K\subseteq \R\times N$ be compact and $\e\in (0,1)$. Then there exists $\delta >0$ such that for all $a,b,c\in I$ with $b-a,c-b \ge \e$ and causal 
$\mathcal{A}_\tau$-minimizers $\eta\colon [a,b]\rightarrow N$, $\gamma \colon [b,c]\rightarrow N$ with $\eta (b)=\gamma (b)$ and $\dot \eta (b)\neq \dot\gamma (b)$ one has 
$$\mathcal{S}_{a}^{c}(\eta(a),\gamma(c))-\mathcal{A}_\tau(\eta)-\mathcal{A}_\tau(\gamma)
\le -\frac{\delta}{|\mathcal{S}_{a}^{c}(\eta(a),\gamma(c))|} |\dot\eta (b)-\dot\gamma(b)|^2.$$
\end{lemma}

\begin{proof}
First notice that it suffices to prove the statement for $(b,\dot\eta(b)),(b,\dot\gamma(b))\in \inte\mathcal{D}_\tau$ since $\delta$ is claimed to be independent of $\gamma$ and $\eta$.
The assertion then follows for $(b,\dot\eta(b)),(b,\dot\gamma(b))\in \mathcal{D}_\tau$ via approximating them by tangent vectors $(b,\dot{\tilde\eta}(b)),
(b,\dot{\tilde\gamma}(b))\in \inte\mathcal{D}_\tau$ and using the continuity of $\mathcal{A}_\tau$ and $\mathcal{S}$.

Choose $\delta>0$ such that 
$$(\partial^2_v L_\tau)_{(s,v)}\ge \frac{\delta}{|L_\tau(s,v)|}\id$$
for all $(s,v)\in \inte \mathcal{D}_\tau$ with $s\in I$ and $v$ based at a point in $K$ according to Lemma \ref{strictconvexity} (ii). 
Cover $I\times K$ with finitely many neighborhoods $V\subset \R\times N$ according to Lemma \ref{normaltau}. Choose $\e>\e'>0$ such that for every 
$(s,x)\in I\times K$ the open set $W:=(s-\e',s+\e')\times B_{\e'}(x)$ is contained in at least one $V$. It suffices to prove the claim for $a',b,c'$ with $b-a'=c'-b=\e'$ 
since concatenating an $\mathcal{A}_\tau$-minimizer from $\eta(a')$ to $\gamma(c')$ with the arcs $\eta|_{[a,a']}$ and $\gamma|_{[c',c]}$ only decreases the left hand side as well as
increases the right hand side. The proof continues to use $a$ and $c$ instead of $a'$ and $c'$ though. 

For $(t,y)\in I^+(a,\eta(a))\cap W$ denote by $Y_{t,y}$ the tangent at $t$ to the unique $\Phi_\tau$-trajectory on $[a,\tau(y)]$ from $\eta(a)$ to $y$. The map 
$(t,y)\mapsto Y_{t,y}$ is smooth for $(t,y)\in I^+((a,\eta(a)))$ by Lemma \ref{phitau} and Proposition \ref{Localdiff}. The inequality
$$\mathcal{S}_{a}^t(\eta(a),\gamma(t))\le \mathcal{A}_\tau(\gamma|_{[b,t]})+\mathcal{S}_{a}^b(\eta(a),\eta(b))$$
for $b\le t\le c$ implies 
$$\partial_t|_{t=s} \mathcal{S}_{a}^t(\eta(a),\gamma(t))\le L_\tau(s,\dot\gamma(s))$$ 
with equality if and only if $\dot\gamma(s)= Y_{s,\gamma(s)}$. One has 
$$\partial_t|_{t=s} \mathcal{S}_{a}^t(\eta(a),\gamma(t))= (\partial_t|_{t=s} \mathcal{S}_{a}^t)(\eta(a),\gamma(s))+(\partial_2 \mathcal{S}_{a}^s)_{(\eta(a),\gamma(s))}(\dot\gamma(s)),$$
i.e. $\partial_t|_{t=s} \mathcal{S}_{a}^t(\eta(a),\gamma(t))$ is an affine function of $\dot\gamma(s)\in TN_{\gamma(s)}$. Thus it must coincide with the tangent to $L_\tau$ at 
$Y_{s,\gamma(s)}$, i.e. 
$$\partial_s\left|_{t=s} \mathcal{S}_{a}^t(\eta(a),\gamma(t))\right.=L_\tau(s,Y_{s,\gamma(s)})
+(\partial_v L_\tau)_{Y_{s,\gamma(s)}}(\dot\gamma(s)-Y_{s,\gamma(s)}).$$ 
Then one has 
\begin{align*}
L_\tau(s,\dot\gamma(s))&\ge L_\tau(s,Y_{s,\gamma(s)})+\partial_v L_\tau(\dot\gamma(s)-Y_{s,\gamma(s)}) +\frac{\delta}{2|L_\text{min}(s)|}|\dot\gamma(s)-Y_{s,\gamma(s)}|^2\\
&= \partial_s \left[\mathcal{S}_{a}^s(\eta(a),\gamma(s))\right]+\frac{\delta}{2|L_\text{min}(s)|}|\dot\gamma(s)-Y_{s,\gamma(s)}|^2.
\end{align*}

\begin{claim}
Denote by $S$ a lower bound of $\mathcal{S}_s^t(x,y)$ for $(s,x),(t,y)\in I\times K$ and set 
$$C_1:=2\exp\left(\frac{1}{\delta}\left(\frac{2S}{\e}\right)^2\right).$$ 
 Further denote by $L_\text{min}(s)$ the minimum of $L_\tau$ on the convex hull of $\dot\gamma(s)$ and $Y_{s,\gamma(s)}$. 
Then there exists a subset $B$ of $[0,\e]$ of measure at least $\e/2$ such that $|L_\text{min}(s)|\le C_1|L_{Y_{s,\gamma(s)}}|$ for all $s\in B$.
\end{claim}

\begin{proof}[Proof of the claim]
For $v, w \in \inte \mathcal{D}_{(s,x)}$ with $v\neq w$ and $(s,x)\in I\times K$ consider 
the convex hull $\conv\{v,w\}$.
Parameterize $\conv\{v,w\}$ by 
$$f\colon \lambda \in [0, |v-w|]\mapsto \left(1-\frac{\lambda}{|v-w|}\right) v+ \frac{\lambda}{|v-w|}w.$$
Next denote by $L_\text{min}$ the minimum of $L\circ f$ and $0\le \lambda_0\le |v-w|$ the parameter achieving this minimum. Let $Y\in\conv\{v,w\}$ and 
denote by $L_Y' :=\left.\frac{d}{d\lambda}\right|_{\lambda=\lambda_1}(L\circ f)(\lambda)$ where $f(\lambda_1)=Y$. Then one has 
\begin{align*}
(L_Y')^2&= |L_Y'||0-L_Y'|= |L_Y'|\cdot \left| \int_{\lambda_1}^{\lambda_0} (L\circ f)''(\lambda)d\lambda\right|\ge |L_Y'|\cdot\delta\left|\int_{\lambda_1}^{\lambda_0}
\frac{1}{|L\circ f(\lambda)|}d\lambda\right|\\
&\ge \delta \int_{\lambda_1}^{\lambda_0} \frac{(L\circ f)'(\lambda)}{|L\circ f (\lambda)|} d\lambda =\delta \log \left| \frac{L_\text{min}}{L_Y} \right|
\end{align*}
where the last two manipulations follow from the convexity of $L\circ f$. 

Now assume that $|L_\text{min}(s)|\ge C_1|L_{Y_{s,\gamma(s)}}|$ on a set $B\subseteq [b,c]$ of measure at least $\e/2$. Then from the first paragraph one has 
$|L_{Y_{s,\gamma(s)}}'|\ge \sqrt{\delta \log C_1}$ on $B$. Note that this implies $L_{Y_{s,\gamma(s)}}'\le 0$ since 
$$0\ge \partial_s \mathcal{S}_{a}^s(\eta(a),\gamma(s))= L_\tau(s,Y_{s,\gamma(s)})+L_{Y_{s,\gamma(s)}}'$$
implies $L_{Y_{s,\gamma(s)}}'\le -L_\tau(s,Y_{s,\gamma(s)})\le -S$ by Remark \ref{localL}. Consequently 
$$\partial_s \mathcal{S}_{a}^s(\eta(a),\gamma(s))= L_\tau(s,Y_{s,\gamma(s)})+L_{Y_{s,\gamma(s)}}'\le L_{Y_{s,\gamma(s)}}' <-\sqrt{\delta \log C_1}$$
which implies 
$$\mathcal{S}_{a}^{c}(\eta(a),\gamma(c))- \mathcal{S}_{a}^b(\eta(a),\gamma(b)) 
\le \int_B \partial_s \mathcal{S}_{a}^s(\eta(a),\gamma(s)) ds \le -\frac{\e}{2} \sqrt{\delta \log C_1}.$$
This constitutes a contradiction to the definition of $C_1$.
\end{proof}

The claim thus implies
\begin{align*}
\mathcal{S}_{a}^{c}(\eta(a),\gamma(c))- &\mathcal{S}_{a}^b(\eta(a),\gamma(b))\le \mathcal{A}_\tau (\gamma)-
\frac{\delta}{2}\int_b^{c} \frac{|\dot\gamma(s)-Y_{s,\gamma(s)}|^2}{|L_\text{min}(s)|}ds\\
&\le \mathcal{A}_\tau (\gamma)-\frac{\delta}{2C_1}\int_B \frac{|\dot\gamma(s)-Y_{s,\gamma(s)}|^2}{|L_{Y_{s,\gamma(s)}}|} ds.
\end{align*}
Next note that the continuity of $\Phi_\tau$ and the invariance of $\partial \mathcal{D}_\tau =L_\tau^{-1}(0)$ under $\Phi_\tau$ implies the existence of a 
$C_2<\infty$ depending only on $I\times K$ such that $(s-a)|L_{Y_{s,\gamma(s)}}|\le C_2 |S_{a}^s(\eta(a),\gamma(s))|$ for all $t\in [b,c]$. Thus one has 
\begin{align*}
\mathcal{S}_{a}^{c}(\eta(a),\gamma(c))- &\mathcal{S}_{a}^b(\eta(a),\gamma(b))
\le \mathcal{A}_\tau (\gamma)-\delta_1\int_B \frac{|\dot\gamma(s)-Y_{s,\gamma(s)}|^2}{|S_{a}^s(\eta(a),\gamma(s))|}ds\\
&\le \mathcal{A}_\tau (\gamma)-\frac{\delta_1}{{|S_{a}^{c}(\eta(a),\gamma(c))|}}\int_B |\dot\gamma(s)-Y_{s,\gamma(s)}|^2 ds.
\end{align*}
Note that again due to the continuity of $\Phi_\tau$ there exists $C_3<\infty$ depending only on $I\times K$ such that for  all $t\in [b,c]$ 
one has $|\dot\gamma(t)-Y_{t,\gamma(t)}|\le C_3 |\dot\gamma(b)-Y_{b,\gamma(b)}|$. This follows from the fact that the image of $Y$ 
is locally invariant under $\Phi_\tau$. Thus there exists $\delta_3>0$ such that 
$$S_{a}^{c}(\eta(a),\gamma(c))- S_{a}^b(\eta(a),\gamma(b))
\le \mathcal{A}_\tau (\gamma)-\frac{\delta_3}{{|S_{a}^{c}(\eta(a),\gamma(c))|}} |\dot\gamma(b)-Y_{b,\gamma(b)}|^2.$$
Finally notice that $Y_{b,\gamma(b)}=\dot\eta(b)$. This finishes the proof.
\end{proof}

\begin{proof}[Proof of Proposition \ref{mathershortening}]
With the same argument as in the preceding proof one can assume that the curves are contained in a normal neighborhood. 

(i) The first step is to show that 
$$(c,x_2(c))\in J^+((a,x_1(a)))$$
under the assumptions in the proposition and for $0<C<\infty$ sufficiently large.
Using the local Lipschitz continuity of 
$$(s,x)\mapsto \mathcal{D}_{(s,x)}=\mathcal{D}_\tau\cap (\{s\}\times TN_x)$$ 
with respect to the Hausdorff distance, which follows directly 
from local Lipschitz continuity of $p\mapsto \mathcal{C}\cap \{d\tau=1\}$ with respect to the Hausdorff distance, there exists 
$C_0<\infty$ only depending on $I\times K$ such that one can choose $(b,\dot\chi_2(b))\in \mathcal{D}_{(b,x_1(b))}$ with 
\begin{align}\label{Etang}
\dist(\dot\chi_2(b),\dot x_2(b))\le C_0\dist(\chi_2(b),x_2(b))=C_0\dist(x_1(b),x_2(b)).
\end{align}
With the smoothness of $\Phi_\tau$ this then implies 
$$\dist(\chi_2(c),x_2(c))\le C_1\dist(x_1(b),x_2(b))$$ 
for some $C_1<\infty$ only depending on $I\times K$ and $\e$. Now the triangle inequality yields 
\begin{align*}
\dist((c,x_2(c)), &\partial J^+((a,x_1(a)))\\
&\ge \dist((c,\chi_2(c)), \partial J^+((a,x_1(a)))-C_1 \dist(x_1(b),x_2(b)).
\end{align*}

Recall from Remark \ref{localL} that 
$$\mathcal{S}_{a}^{c}(x_1(a),\chi_2(c))^2=L((\exp^\mathbb{L}_{(a,x_1(a))})^{-1}((c,\chi_2(c))))^2$$
and that $L^2$ is smooth up to and beyond $\partial\mathcal{D}_\tau$. Thus there exists $\delta_0>0$, only depending on $I\times K$ and $\e$, such that 
$$\dist((c,\chi_2(c)), \partial J^+((a,x_1(a)))\ge \delta_0 \mathcal{S}_{a}^{c}(x_1(a),\chi_2(c))^2.$$
Now Lemma \ref{actionestimate} implies $\mathcal{S}_{a}^{c}(x_1(a),\chi_2(c))^2\ge \delta_2 \dist(\dot{x}_1(b),\dot{\chi}_2(b))^2$
as a special case. With \eqref{Etang} one then obtains 
$$\mathcal{S}_{a}^{c}(x_1(a),\chi_2(c))^2\ge \delta_3 \dist(\dot{x}_1(b),\dot{x}_2(b))^2$$
for some $\delta_3>0$ if $C$ is sufficiently large. Consequently one has
$$\dist((c,x_2(c)), \partial J^+((a,x_1(a)))\ge \frac{1}{2}\dist((c,\chi_2(c)), \partial J^+((a,x_1(a)))$$
for $\dist(\dot{x}_1(b),\dot{x}_2(b))^{2}\ge C\dist(x_1(b),x_2(b))$ with $C$ sufficiently large. Thus one concludes $(c,x_2(c))\in J^+((a,x_1(a)))$.

(ii) Remark \ref{localL} implies that 
$$\mathcal{S}_{a}^{c}(x_1(a),x_2(c))^2\ge \delta_4 \dist((c,x_2(c)),J^+((a,x_1(a))))$$
for some $\delta_4>0$ depending only on $I\times K$ and $\e$, since the fiber derivative of $L_\tau^2$ does not vanish anywhere 
on $\partial \mathcal{D}_\tau$. Thus one has 
$$\mathcal{S}_{a}^{c}(x_1(a),x_2(c))^2\ge \frac{\delta_4\delta_0}{2} \mathcal{S}_{a}^{c}(x_1(a),\chi_2(c))^2.$$
With the convexity of $\mathcal{S}$ one then concludes
\begin{equation}\label{actionerror}
\begin{split}
|\mathcal{S}_{a}^{c}(x_1(a),x_2(c))-\mathcal{S}_{a}^{c}(x_1(a),\chi_2(c))|\le
 \frac{C_2}{|\mathcal{S}_{a}^{c}(x_1(a),\chi_2(c))|}\dist(x_1(b),x_2(b))
\end{split}
\end{equation}
for some $C_2<\infty$ depending only on $I\times K$ and $\e$. Finally one has 
\begin{align*}
\mathcal{S}_{a}^{c}(x_1(a),x_2(c))-&\mathcal{S}_{a}^{b}(x_1(a),x_1(b))-\mathcal{S}_{b}^{c}(x_2(b),x_2(c))\\
\le&\; \mathcal{S}_{a}^{c}(x_1(a),\chi_2(c))-\mathcal{S}_{a}^{b}(x_1(a),x_1(b))-\mathcal{S}_{b}^{c}(\chi_2(b),\chi_2(c))\\
&+|\mathcal{S}_{a}^{c}(x_1(a),\chi_2(c))-\mathcal{S}_{a}^{c}(x_1(a),x_2(c))|\\
&+|\mathcal{A}_\tau(x_2|_{[b,c]})-\mathcal{A}_\tau(\chi_2)|.
\end{align*}
The first term on the right hand side is bounded from above by 
$$-\frac{\delta}{|\mathcal{S}_{a}^{c}(x_1(a),\chi_2(c))|}\dist(\dot{x}_1(b),\dot{x}_2(b))^2$$
according to Lemma \ref{actionestimate} and the choice of $\chi_2$ for some $\delta>0$. The second term is bounded from above by 
$$\frac{C_2}{|\mathcal{S}_{a}^{c}(x_1(a),\chi_2(c))|}\dist(x_1(b),x_2(b))$$
according \eqref{actionerror}. Finally one has 
$$|\mathcal{A}_\tau(x_2|_{[b,c]})-\mathcal{A}_\tau(\chi_2)|\le C_3\dist(x_1(b),x_2(b))$$ 
according to \eqref{Etang} for some $C_3<\infty$. Now note that 
$$\dist(\dot{x}_1(b),\dot{\chi}_2(b))\ge \dist(\dot{x}_1(b),\dot{x}_2(b))-C_0\dist(x_1(b),x_2(b))$$
by the triangle inequality and \eqref{Etang}. This implies that 
\begin{align*}
\mathcal{S}_{a}^{c}(x_1(a),x_2(c))-\mathcal{S}_{a}^{b}(x_1(a),x_1(b))-\mathcal{S}_{b}^{c}(x_2(b),x_2(c))
\le -\frac{\delta}{2}\dist(\dot{x}_1(b),\dot{x}_2(b))
\end{align*}
for $C<\infty$ sufficiently large.

Repeating the arguments for $\mathcal{S}_{a}^{c}(x_2(a),.)$, $x_2|_{[a,b]}$ and $x_1|_{[b,c]}$ one obtains
\begin{align*}
\mathcal{S}_{a}^{c}(x_1(a),x_1{c})+\mathcal{S}_{a}^{c}(x_2(a),x_2(c))-\mathcal{S}_{a}^c(x_1(a),x_2(c))&-\mathcal{S}_{a}^c(x_2(a),x_1(c))\\
&\le -\delta \dist(\dot{x}_1(b),\dot{x}_2(b))^2 
\end{align*}
for $C<\infty$ sufficiently large. The claim is now immediate for $y_1\colon [a,c]\to N$ the 
$\mathcal{A}_\tau$-minimizer from $x_1(a)$ to $x_2(c)$ and $y_2\colon [a,c]\to N$ the $\mathcal{A}_\tau$-minimizer from $x_2(a)$ to $x_1(c)$.
\end{proof}

\begin{proof}[Proof of Theorem \ref{intermediateregularity}]
Let $\Pi$ be a dynamical optimal coupling of $\mu$ and $\nu$. For $k\in\N$ consider the subcoupling
$$\Pi_k:= \Pi|_{\{\gamma|\; \tau(\gamma(1))-\tau(\gamma(0))\ge 1/k\}}.$$
Since the supports of $\mu$ and $\nu$ are disjoint one knows that for every compact set $I\times K\subseteq \R\times N\cong M$ there exists
$k$ with
$$(\ev\nolimits_t)_\sharp \Pi|_{I\times K}\equiv (\ev\nolimits_t)_\sharp \Pi_k|_{I\times K}$$
for all $t\in [0,1]$.  Fix $I\times K\subset M$ compact and $k\in\N$ such that $\gamma\in \supp\Pi_k$  for all $\gamma\in \supp\Pi$ with $\gamma\subset I\times K$.

Consider the reparameterization $\eta\colon [\tau(\gamma(0)),\tau(\gamma(1))]\to M$ of $\gamma\in \supp \Pi_k$ such that $\tau\circ \eta(s)=s$. 
Next let $\e_0>0$ be given and consider the restriction of $\gamma\in \supp\Pi_k$ to $[\e_0,1-\e_0]$. 
Then there exists $\e_1>0$ only depending on $\e_0$ such that $|s(t)-\tau(\gamma(0))|,|s(t)-\tau(\gamma(1))|\ge 2\e_1$ for all $\gamma\in \supp\Pi_k$ and
the reparameterization $\eta$ with $\eta(s(t))\equiv\gamma(t)$ and $t\in [\e_0,1-\e_0]$.

Let $(\gamma_1,t_1),(\gamma_2,t_2)\in \supp\Pi_k\times[\e_0,1-\e_0]$. Denote with 
$$\eta_i\colon [\tau(\gamma_i(0)),\tau(\gamma_i(1))]\to M$$ 
the reparameterization of $\gamma_i$ as in the previous paragraph. Since $\tau$ is Lipschitz on $I\times K$ with constant $L<\infty$, i.e.
$$|\tau(\gamma_1(t_1))-\tau(\gamma_2(t_2))|\le L\dist(\gamma_1(t_1),\gamma_2(t_2))$$
one has $|b_2-b_1|\le L\dist(\gamma_1(t_1),\gamma_2(t_2))$ for $b_i=\tau(\gamma_i(t_i))$ and $i=1,2$. 
For $\dist(\gamma_1(t_1),\gamma_2(t_2))$ smaller than $\e_1/L$ one has thus 
$$b_1-\tau(\gamma_2(0)),\tau(\gamma_2(1))-b_1\ge \e_1.$$
Therefore $\eta_2$ is well defined on $[b_1-\e_1,b_1+\e_1]$.
With the smoothness of $\Phi_\tau$ one concludes that there exists a constant $C_0<\infty$ depending only on $\e_0$ such that
$$\dist(\dot\eta_2(b_1),\dot\eta_2(b_2))\le C_0 \dist(\gamma_1(t_1),\gamma_2(t_2)).$$
Further the triangle inequality implies that 
$$\dist(\dot\eta_2(b_1),\dot\eta_1(b_1))\ge \dist(\dot\eta_1(b_1),\dot\eta_2(b_2))-C_0\dist(\gamma_1(t_1),\gamma(t_2))$$
and 
$$\dist(\eta_2(b_1),\eta_1(b_1))\le C_1\dist(\gamma_1(t_1),\gamma_2(t_2))$$
for some $C_1<\infty$ depending only on $I\times K$.

Now choose $\delta,\kappa>0$ and $C<\infty$ for $\e:=\e_1$ and $I\times K$ according to Proposition \ref{mathershortening}.
For $C_2<\infty$ sufficiently large assuming that
$$C_2\dist(\gamma_1(t_1),\gamma_2(t_2))< \dist([\dot{\gamma}_1(t_1)],[\dot{\gamma}_2(t_2)])^2$$
and $\dist(\gamma_1(t_1),\gamma_2(t_2))\le \min\{\delta/C_1,\e_1/L\}$ one has 
$$C \dist(\eta_2(b_1),\eta_1(b_1))< \dist(\dot\eta_2(b_1),\dot\eta_1(b_1))^2.$$
Then Proposition \ref{mathershortening} implies that
\begin{align*}
&c_L(\eta_1(b_1-\e),\eta_2(b_1+\e))+c_L(\eta_2(b_1-\e),\eta_1(b_1+\e))\\
&-c_L(\eta_1(b_1-\e),\eta_1(b_1+\e))-c_L(\eta_2(b_1-\e),\eta_2(b_1+\e))<0.
\end{align*}
With the triangle inequality for $c_L$ follows
$$c_L(\gamma_1(0),\gamma_2(1))+c_L(\gamma_2(0),\gamma_1(1))-c_L(\gamma_1(0),\gamma_1(1))-c_L(\gamma_2(0),\gamma_2(1))<0$$
clearly contradicting the cyclic monotonicity of the optimal coupling $(\ev_0,\ev_1)_\sharp\Pi$ of $\mu$ and $\nu$, see Proposition \ref{propcycmon}.
Thus there exists $D<\infty$ with 
$$ \dist([\dot{\gamma}_1(t_1)],[\dot{\gamma}_2(t_2)])^2\le D\dist(\gamma_1(t_1),\gamma_2(t_2))$$
showing the injectivity of the projection and the H\"older continuity of the inverse.
\end{proof}

A $C^2$-function $L_T\colon \R\times TN\to \R$ is a {\it Tonelli-Lagrangian}, see \cite{bebu2}, if for all $(t,x)\in \R\times N$
\begin{itemize}
\item[(i)] the restriction $L_T|_{\{t\}\times TN_x}$ is convex with positive definite Hessian everywhere,
\item[(ii)] $L_T(t,v)/|v|\to \infty$ as $|v|\to \infty$ for $v\in TN_x$ and
\item[(iii)] the Euler-Lagrange flow of $L_T$ is complete.
\end{itemize}

\begin{proof}[Proof of Theorem \ref{lipreg}]
Choose a compact set ${K}'\subseteq \inte\mathcal{D}_\tau$ such that $K\subseteq \inte {K}'$ the interior of ${K}'$. Next construct a 
Tonelli-Lagrangian $L_T\colon\R\times TN\to \R$ with $L_T\ge L_\tau$ and $L_T|_{{K}'}\equiv L_\tau|_{{K}'}$. Then every $\mathcal{A}_\tau$-minimizer 
$\gamma$ with $\dot\gamma\in {K}'$ is also a minimizer for the action induced by $L_T$. Now the claim follows from the classical regularity result for Tonelli-Lagrangians, e.g.
\cite[Theorem A]{bebu2}. 
\end{proof}

\subsection{Proof of Theorem \ref{Thmmonge} and \ref{Thmmonge2}}\label{secmonge}

First the proof of Theorem \ref{Thmmonge} is given. After that the necessary changes to the argument for the proof of Theorem 
\ref{Thmmonge2} are indicated.

The proof of Theorem \ref{Thmmonge} is essentially carried out via two propositions. 
\begin{prop}\label{mongeprop1}
Let $(\mu,\nu)\in\mathcal{P}_\tau^+(M)$. Assume that $\mu$ and $\nu$ are concentrated on a locally uniformly spacelike hypersurface 
$A$ and an achronal set $B$, respectively. Further assume that $\mu$ is absolutely continuous with respect to the Lebesgue measure 
on $A$ and that $\supp\mu$ is $\nu$-neglectable. Then for every optimal coupling $\pi$ of $\mu$ and $\nu$ there exists a set 
$R\subset M\times M$ of full $\pi$-measure such that for all $(x,y_1),(x,y_2)\in R$ there exists an $\mathcal{A}$-minimizer 
$\gamma\colon [0,1]\to M$ containing $x,y_1,y_2$ in its trace.
\end{prop}

The proof of Proposition \ref{mongeprop1} needs the following lemma. Recall that $x\in N$ is a {\it Lebesgue point} of a set $C\subset N$ if 
$$\lim_{\delta\to 0} \frac{\mathcal{L}^n(C\cap B_\delta(x))}{\mathcal{L}^n(B_\delta(x))}=1,$$
where $B_\delta(y)$ continues to denote the metric ball of radius $\delta>0$ around $x$.

\begin{lemma}\label{lemleb}
Let $N$ be a manifold, $\mu,\nu\in \mathcal{P}(N)$, $\pi\in \Pi(\mu,\nu)$ and $\Sigma$ a $\sigma$-compact set such that $\pi(\Sigma)=1$. 
Assume that $\mu$ is absolutely continuous with respect to the Lebesgue measure on $N$. Then $\pi$ is concentrated on a $\sigma$-compact set $R(\Sigma)$ 
such that for all $(x,y)\in R(\Sigma)\subset N\times N$ the point $x$ is a Lebesgue point of $\pi_1(\Sigma\cap (N\times \overline{B_r(y)}))$ for all $r>0$. 
\end{lemma}

A version for the case $M= \R^n$ is proved in \cite[Lemma 4.3]{chpa}. The proof carries over mutatis mutandis to the present situation of manifolds.

\begin{proof}[Proof of Proposition \ref{mongeprop1}]
Let $\Pi$ be a dynamical optimal coupling of $\mu$ and $\nu$. Then $\pi :=(\ev_0,\ev_1)_\sharp \Pi$ is an optimal coupling of $\mu$ and $\nu$. 
One can assume that $\mathcal{A}$-minimizers between points in $\supp\mu$ and $\supp \nu$ are unique up to parameterization. This 
can be seen as follows. By passing to a dynamical subcoupling $\Xi'$, according to Corollary \ref{cor_interpolation}, one can first assume 
that $\supp\mu$ is compact. The proof continues to use the notation $\Pi$ for the dynamical optimal coupling.
By Corollary \ref{intermcoup} one can assume that the $\mathcal{A}$-minimizers between points in $\supp \mu$ and $\supp\nu$ are unique 
up to parameterization by considering the transport $\pi_{0,\sigma}$ between $0$ and $\sigma\colon \Gamma\to [0,1]$ with $0<\sigma(\gamma)$ sufficiently small. 
One can choose $\sigma$ such that $\supp\nu$ is compact.

Note that
$$0=\nu(\supp \mu)=\pi(\supp\mu\times \supp\nu \cap \triangle)$$ 
where $\triangle$ denotes the diagonal in $M\times M$. Thus $\Pi$-almost every $\mathcal{A}$-minimizer is nonconstant. The assumption that $\mu$ is concentrated 
on a locally uniformly spacelike hypersurface implies that every nonconstant causal curve can intersect $\supp\mu$ at most once. Therefore $\supp\mu$ is 
$(\ev_t)_\sharp \Pi$-neglectable for all $t>0$.

Note that since $\mu$ and $\nu$ are supported on Lipschitz graphs over $N$ one can consider both measures to be 
supported on $N$ without losing the absolute continuity of $\mu$ with respect to the Lebesgue measure. Therefore 
one can apply Lemma \ref{lemleb} to $\mu$ and $\nu$ seen as measures on $N$ and obtain a set $R\subseteq A\times B$ 
by revoking the identification via the graphs. Choose a set $R=R(\Sigma)\subseteq A\times B$ according to Lemma \ref{lemleb}. 

Assume that there exist $(x,y_1)$ and $(x,y_2)\in R$ such that $y_i$ does not lie on the $\mathcal{A}$-minimizer between $x$ and $y_j$ for $i\neq j$. 
Then one knows that the tangents $\dot{\gamma}_i(0)$ to the $\mathcal{A}$-minimizers $\gamma_i\in \Gamma_{x\to y_i}$ are not parallel. 
Choose a diffeomorphism $\psi$ from the unit ball $B_1(0)$ in $\R^m$ to a neighborhood $U$ of $x$ with $\psi(0)=x$. For $\delta >0$ 
define $\psi_\delta\colon B_1(0)\rightarrow U$, $v\mapsto \psi(\delta v)$. 

It is obvious that $L_\delta := \frac{1}{\delta}\psi_\delta^*(L)$ converges for $\delta \rightarrow 0$ to $L|_{\mathcal{C}_x}$ uniformly 
on compact subsets of $\inte \mathcal{C}$ in any $C^k$-topology. Especially the minimizers of the action induced by $L_\delta$ converge uniformly to 
straight lines in $B_1(0)$. 

Next choose sequences $\delta_n,r_n\downarrow 0$ such that 
\begin{align}\label{lemleb2}
\lim_{n\to\infty} \frac{\mathcal{L}_A(\pi_1(\Sigma \cap (A\times \overline{B_{r_n}(y_2)}))\cap \text{im}(\psi_{\delta_n}))}{\mathcal{L}_A(A\cap \text{im}(\psi_{\delta_n}))}=1
\end{align}
where $\mathcal{L}_A$ denotes the Lebesgue measure on $A$.
Since the distance from $x$ to $y_2$ can be bounded from below and due to the structure of the $\psi_\delta$'s one concludes that 
the tangents $\dot\eta$ at $x$ converge to $\dot\gamma_2(0)$ at $x$ for $\mathcal{A}$-minimizers $\eta\in\Gamma$ 
connecting a point in $\text{im}\psi_{\delta_n}$ with a point in $B_{r_n}(y_2)$. 
Further by \eqref{lemleb2} one can choose points $(w_n,z_n)\in \text{im}\psi_{\delta_n}\times B_{r_n}(y_2)$ with $(w_n,z_n)\in \supp\pi$, 
$\dist(\psi_{\delta_n}^{-1}(x),\psi_{\delta_n}^{-1}(w_n))\ge \frac{1}{2}$ and $\psi_{\delta_n}^{-1}(w_n)\to a\dot{\gamma}_1(0)
+b\dot{\gamma}_2(0)$ with $a,b\in \R$ and $b<0$. Thus the $L|_{\mathcal{C}_x}$-minimizer $t\mapsto t\cdot \dot{\gamma}_1(0)$ and 
$t\mapsto v+ t\cdot \dot{\gamma}_2(0)$ intersect for some positive value of $t$. A simplified version of 
Proposition \ref{mathershortening} now shows that this crossing can be shortened by a nonzero amount. Since the convergence is 
uniform a fraction of this shortening survives when passing to $L_{\delta_n}$ for $n$ sufficiently large. This now contradicts the cyclic 
monotonicity of the optimal coupling.
\end{proof}

Consider the set ${I}'_B$ of $\mathcal{A}$-minimizers $\gamma'\in \Gamma$ which intersect $B$ in more than one point. Note that 
$\dot{\gamma}'\in \partial\mathcal{C}$ for all $\gamma'\in {I}'_B$ since $B$ is achronal. Identify $M$ with $\R\times N$ via the splitting 
$\tau$ as in Section \ref{intermregu}. Define the set $I_B$ to be the set of reparameterizations $\gamma$ of $\gamma'\in {I}'_B$  with 
$\tau\circ \gamma=\id$. Then the curves in ${I}_B$ correspond one-to-one with $\Phi_\tau$-orbits in $N$. Denote the set of these $\Phi_\tau$-orbits
by $I_B$ as well.

\begin{prop}\label{mongeprop2}
If $A$ is a locally uniformly spacelike hypersurface and $B$ is achronal then the set formed by 
the intersections of orbits in $I_B$ with $A$ is $\mathcal{L}_A$-neglectable. 
\end{prop}

Assume for the moment that (i) $A$ is uniformly spacelike, (ii) $B$ is precompact and (iii) the distance between the first and 
the last intersections of $\mathcal{A}$-minimizers with $B$ is uniformly bounded from below. 
Let $(y_1,y_2)\in J^+ \cap (B\times B)$ and $\gamma\in I_B$ be an $\mathcal{A}$-minimizer between $y_1$ and $y_2$. Choose $\delta>0$ such that 
$\tau(B_\delta(y_1))$ and $\tau(B_\delta(y_2))$ are disjoint. Choose $b\in \R$ between $\tau(B_\delta(y_1))$ and $\tau(B_\delta(y_2))$. 
Denote by $S_B$ the set of intersections of curves $\gamma\in I_B$ with $\{b\}\times N$ and let $\mathcal{L}_{\{b\}\times N}$ denote
the Lebesgue measure on $\{b\}\times N$.

\begin{lemma}\label{mongeNeg}
$S_B$ is a $\mathcal{L}_{\{b\}\times N}$-neglectable set.
\end{lemma}

\begin{proof}
Consider $\eta\in \Gamma$ with endpoints in $B_{\delta}(y_1)\cap B$ and $B_{\delta}(y_2)\cap B$. Denote the intersection 
of $\eta$ with $\{b\}\times N$ by $z$. Choose a convex neighborhood $U$ around $z$ disjoint from $B_{\delta}(y_1)\cap B$ and 
$B_{\delta}(y_2)\cap B$. Denote by $\eta_\alpha$ and $\eta_\omega$ the initial and the terminal point on $\eta$ in $U$, respectively. Then one has
$$S_B\cap U \subseteq J^-(\eta_\omega)^c\cap J^+(\eta_\alpha)^c=(J^-(\eta_\omega)\cup J^+(\eta_\alpha))^c.$$
With the same argument as in the proof of Proposition \ref{mongeprop1} one can assume, after possibly restricting $U$, that 
$(\{b\}\times N)\cap J^+(\eta_\alpha)$ and $(\{b\}\times N)\cap J^-(\eta_\omega)$ are strictly convex sets.
Thus there exists $r>0$ such that for every point $z_0\in S_B$  there exist two two points $z_1,z_2\in \{b\}\times N$ with 
$$B_r(z_1)\subset(\{b\}\times N)\cap J^+(\eta_\alpha),\; B_r(z_2)\subset (\{b\}\times N)\cap J^-(\eta_\omega)$$
and 
$$B_r(z_1)\cap B_r(z_2)=\{z_0\}.$$
Therefore for every $\e>0$ $S_B$ can be covered by at most $\e^n$ disjoint sets with volume less than $\e^{n+1}$. 
This shows that $S_B$ is $\mathcal{L}_{\{b\}\times N}$-neglectable. 
\end{proof}

\begin{lemma}\label{mongeLip}
The map $S_B\to TN\cong \{b\}\times TN$ mapping $z\in S_B$ to the tangent vector in $\mathcal{D}_z$ of an $\mathcal{A}_\tau$-minimizer in $I_B$ intersecting $z$ is well defined and Lipschitz.
\end{lemma}

\begin{proof}
Let $\gamma_1$ be a $\mathcal{A}$-minimizer between $y_1,y_2\in B$ and $\gamma_2$ be a $\mathcal{A}$-minimizer between $y_3, y_4\in B$ 
that meet at an intermediate point $z$ with different tangent vectors. Then $y_2\in I^+(y_3)$ 
and $y_4\in I^+(y_1)$. Both induces a contradiction to the achronality of $B$. Thus the map is well defined.

Now let $x,z\in S_B$ and $\gamma_x,\gamma_z\in I_B$ containing $x$ and $z$ in their traces, respectively. Choose $y_1,y_2\in B$ 
such that $\gamma_z$ connects $y_1$ and $y_2$. Then one has $x\in I^+(y_1)^c\cap I^-(y_2)^c$ by the achronality of $B$. Therefore 
$$\dist(x,I^+(y_1)),\dist(x,I^-(y_2))\le C_1\dist(x,z)^2$$
for some $C_1<\infty$ depending  only $\supp\mu \cup \supp\nu$. 

Choose $w\in \partial J^+(y_1)\cap (\{b\}\times N)$ the nearest point to $x$. Then there exists $C_2<\infty$ and $(b,\dot\chi(b))\in \mathcal{D}_w$ with 
\begin{align}\label{mongelipeq}
\dist((b,\dot\chi(b)),\dot\gamma_x(b))\le C_2\dist(w,z).
\end{align}
Recall that one has assumed that the distance between the intersections of $\mathcal{A}$-minimizers with $B$ is bounded from below. Therefore
there exists $\e>0$ such that 
$$\max \tau|_{B_\delta(y_1)}<b-\e<b+\e<\min\tau|_{B_\delta(y_2)}.$$
Then by Lemma \ref{actionestimate} there exists $\delta_1>0$ only depending on $\supp\mu \cup \supp\nu$ such that 
$$c_L(\gamma_z(b-\e),(b+\e,\chi(b+\e)))^2\ge \delta_1\dist(\dot\chi(b), Y_w)^2$$
where $Y_w$ denotes the tangent to the unique $\mathcal{A}_\tau$-minimizer $\eta_w\colon [b-\e,b]\to N$ whose graph connects $\gamma_z(b-\e)$ and $w$.
Further since $c_L(\gamma_z(b-\e),.)^2$ is Lipschitz up to the boundary of its domain there exists $\delta_2>0$ with 
$$\dist((b+\e,\chi(b+\e)),\partial J^+(\gamma_z(b-\e)))\ge \delta_2 c_L^2(\gamma_z(b-\e),(b+\e,\chi(b+\e))).$$
By the triangle inequality and \eqref{mongelipeq} one has 
\begin{align*}
\dist(\dot\chi(b),Y_w)&\ge \dist(Y_w,\dot\gamma_x(b))-C_2\dist(w,x)\\
&\ge \dist(\dot\gamma_z(b),\dot\gamma_x(b))-C_3\dist(w,z)-C_2\dist(w,x)
\end{align*}
where the last inequality follows from the Lipschitz continuity of the vector field $u\mapsto Y_u$. Since $\dist(w,x)\le \dist(z,x)$
one concludes
$$\dist(\dot\chi(b),Y_w)\ge  \dist(\dot\gamma_z(b),\dot\gamma_x(b))-(2C_3+C_2)\dist(z,x)\ge \frac{1}{2}\dist(\dot\gamma_z(b),\dot\gamma_x(b)).$$
if $\dist(\dot\gamma_z(b),\dot\gamma_x(b))\ge 2(2C_3+C_2) \dist(z,x)$. Now the triangle inequality and the last estimate imply that 
\begin{align*}
\dist&(\gamma_x(b+\e),\partial J^+(\gamma_z(b-\e)))\\
&\ge \dist((b+\e,\chi(b+\e)),\partial J^+(\gamma_z(b-\e)))-\dist(\gamma_x(b+\e),(b+\e,\chi(b+\e)))\\
&\ge \delta_3 \dist(\dot\gamma_z(b),\dot\gamma_x(b))^2-\dist(\gamma_x(b+\e),(b+\e,\chi(b+\e)))
\end{align*}
for some $\delta_3>0$. Next by the continuity of $\Phi_\tau$ one has 
\begin{align*}
\dist(\gamma_x(b+\e),\partial J^+(\gamma_z(b-\e)))&\ge \delta_3 \dist(\dot\gamma_z(b),\dot\gamma_x(b))^2-C_4\dist(w,x)\\
&\ge \delta_4 \dist(\dot\gamma_z(b),\dot\gamma_x(b))^2-C_{4}C_1\dist(z,x)^2
\end{align*}
for some $C_4<\infty$. Now if 
$$\dist(\dot\gamma_z(b),\dot\gamma_x(b))\ge \max\left\{2(2C_3+C_2),\sqrt{\frac{C_4C_1}{\delta_4}}\right\} \dist(z,x)$$
one concludes $\gamma_x(b+\e)\in I^+(\gamma_z(b-\e))$. This in turn implies that the endpoint of $\gamma_x$ in $B$ is 
contained in $I^+(y_2)$, clearly a contradiction to the achronality of $B$.
\end{proof}

\begin{proof}[Proof of Proposition \ref{mongeprop2}]
Since a countable union of neglectable sets is neglectable one makes a few simplifying assumptions. One assumes that (i) $A$ is uniformly spacelike, i.e.
the distance of $TA\cap T^1M$ from $\mathcal{C}^1$ is bounded away from $0$, (ii) $B$ is 
precompact and (iii) the distance between the first and the last intersections of $\mathcal{A}$-minimizers with $B$ is uniformly bounded from below. 

By Lemma \ref{mongeNeg} the set $S_B$ is $\mathcal{L}_{\{b\}\times N}$-neglectable. 
Further by Lemma \ref{mongeLip} the map that assigns to each intersection point the tangent of the corresponding $\mathcal{A}_\tau$-minimizer is Lipschitz. Choose a 
Lipschitz extension of this map to $N$ according to Kirzbraun's Theorem, cf. \cite[Theorem 1.31]{schwartz}. Then the unique intersection 
of $\mathcal{A}_\tau$-minimizers in $I_B$ with $A$ is the image of a $\mathcal{L}_{\{b\}\times N}$-neglectable set under a Lipschitz map. Therefore it is 
$\mathcal{L}_A$-neglectable.
\end{proof}

\begin{proof}[Proof of Theorem \ref{Thmmonge}]
First one shows that any optimal coupling is concentrated on the graph of a map. Any such map is measurable since 
couplings are Borel measures. Choose a dynamical optimal coupling $\Pi$. Denote by $\Pi_\triangle$ the restriction of $\Pi$ to
the set of constant $\mathcal{A}$-minimizers and $\Pi_C:=\Pi-\Pi_\triangle$. Further set $\mu_\triangle:=(\ev_0)_\sharp \Pi_\triangle$ and $\mu_C:=(\ev_0)_\sharp \Pi_C$.
Note that by construction one has $\mu =\mu_\triangle +\mu_C$. 

First one shows that $\supp\mu_\triangle\cap \supp\mu_C$ is a $\mathcal{L}_A$-neglectable set. To this end note that $\supp\mu_\triangle \subseteq \supp \nu$
since $\mu_\triangle$ is induced by constant curves. Now if $x\in\supp\mu_\triangle\cap \supp\mu_C$, $x$ is contained in $\supp\nu$ and there exists 
$y\in\supp\nu \cap J^+(x)\setminus\{x\}$. So $x$ lies on an $\mathcal{A}$-minimizer that intersects the support of 
$\nu$ at least twice. The set consisting of such points was shown in Proposition \ref{mongeprop2} to be $\mathcal{L}_A$-neglectable which implies the initial claim.

Assume for the moment that $\pi_\triangle :=(\ev_0,\ev_1)_\sharp \Pi_\triangle$ and $\pi_C :=(\ev_0,\ev_1)_\sharp \Pi_C$ are separately concentrated on a graph.
Then $\pi$ is concentrated on the union of these graphs since $\pi =\pi_\triangle+\pi_C$. The overlap of these graphs lies in 
$\pi_1^{-1}(\supp\mu_\triangle\cap \supp\mu_C)$. Since $\supp\mu_\triangle\cap \supp\mu_C$ is $\mathcal{L}_A$-neglectable it is also $\mu$-neglectable and 
therefore $\pi_1^{-1}(\supp\mu_\triangle\cap \supp\mu_C)$ is $\pi$-neglectable. Thus $\pi$ is concentrated on a graph.

Therefore one has to show that $\pi_\triangle$ and $\pi_C$ are concentrated separately on a graph. This claim is trivial for $\pi_\triangle$ since $\pi_\triangle$ is concentrated on the 
diagonal of $M\times M$. 

For $\pi_C$ note that by construction $\pi_C (\triangle)=0$. Since $\pi_C(\triangle)
\ge \nu(\supp\mu)$ one can apply Proposition \ref{mongeprop1} to the situation of $\mu_C$ and $\nu_C:=(\ev_1)_\sharp \Pi_C$ with the coupling $\pi_C$.
Assume first that there exists a set $S\subset M$ with $\mu_C(S)>0$ such that for every $x\in S$ there exist
$y_1, y_2\in\supp \nu_C$ with $y_1\neq y_2$, $(x,y_i)\in \supp\pi_C$ and no $\mathcal{A}$-minimizer from $x$ to $y_i$ meets $y_j$ for $i\neq j$. 
By the martingale property of $\pi_C$ one has $\pi_C(\pi_1^{-1}(S))=\mu_C(S)>0$. Now for the set $R$ constructed in Proposition \ref{mongeprop1} 
one has $R\cap \pi_1^{-1}(S)\neq \emptyset$. But this contradicts the property of $R$ given in Proposition \ref{mongeprop1}. Therefore the set of 
points transported into two different directions is $\mu_C$-neglectable. 

It remains to show that the set transported along one $\mathcal{A}$-minimizer, but to at least two points in $B$ is $\mathcal{L}_A$-neglectable. 
But this is the content of Proposition \ref{mongeprop2} since $\mu_C$ is absolutely continuous with respect to $\mathcal{L}_A$. This follows 
directly from the assumption that $\mu$ is absolutely continuous with respect to $\mathcal{L}_A$.

Uniqueness of the optimal coupling follows from the observation that if two optimal couplings exist, any convex combination of both is 
optimal as well. But any nontrivial convex combination of two couplings, concentrated on separate graphs, is not concentrated on a graph unless they coincide.
\end{proof}

The proof of Theorem \ref{Thmmonge2} differs only in minor details from that of Theorem \ref{Thmmonge}. These modifications are indicated in the following.

\begin{prop}\label{mongeprop3}
Let $\mu,\nu\in \mathcal{P}(M)$ be as in the assumptions of Theorem \ref{Thmmonge2} and assume that $\supp\mu$ is $\nu$-neglectable. Then for 
every optimal coupling $\pi$ of $\mu$ and $\nu$ there exists a set $R$ of full $\pi$-measure such that for all $(x,y_1),(x,y_2)\in R$ there 
exists a $\mathcal{A}_\tau$-minimizer $\gamma$ containing $x,y_1,y_2$ in its trace.
\end{prop}

\begin{proof}
Choose a dynamical optimal coupling $\Pi$ between $\mu$ and $\nu$.
Like in Proposition \ref{mongeprop1} one can assume that $\mathcal{A}$-minimizers between $\supp\mu$ and $\supp\nu$ are nonconstant and unique up to parameterization.
To see this first consider dynamical subcouplings $\Xi'$ instead of $\Pi$, according to Corollary \ref{cor_interpolation}, for the restriction of $\mu$ to 
$\overline{B_{r}(p)}\cap B_\e(\supp\nu)^c$ for $p\in\supp\mu$, $\e>0$ and $r<\inj(L)/2$, where $\inj(L)$ denotes the injectivity radius of $L$ on a sufficiently large 
compact subset of $M$. The proof continues to use the notation $\Pi$ for the dynamical optimal coupling. 

Since the distance between $\supp\mu$ and $\supp\nu$ is positive $\Pi$-almost all $\mathcal{A}$-minimizers are 
nonconstant. Consequently one can choose a measurable function $\sigma\colon \Gamma\to (0,1]$ such that $\gamma(\sigma(\gamma))\in (\supp\mu)^c \cap 
B_{\inj(L)}(\gamma(0))$ for $\Pi$-almost all $\gamma\in \Gamma$. The resulting restriction is optimal according to Corollary \ref{intermcoup}. By construction
one knows that $\supp\mu$ is $(\ev\circ(\id\times \sigma))_\sharp \Pi$-neglectable. 

Now one applies Lemma \ref{lemleb} to $\mu$ and $\nu$ to obtain the set $R\subset M\times M$. The remainder of the argument is absolutely analogous. 
\end{proof}

Recall that $I_B$ denotes the set of $\mathcal{A}$-minimizers $\gamma\colon I\to M$ with $\tau\circ \gamma=\id$ which intersect $B$ in more than one points.

\begin{prop}\label{mongeprop4}
The set formed by the traces of orbits in $I_B$ is $\mathcal{L}_M$-negletable.
\end{prop}

\begin{proof}
As before one can assume that $B$ is precompact and the distance between two intersections of a $\mathcal{A}$-minimizer with $B$ is uniformly bounded from below. 
Then as above Lemma \ref{mongeNeg} and \ref{mongeLip} apply to the present case as well with the same notation. Choose a Lipschitz extension of the 
Lipschitz map obtained in Lemma \ref{mongeLip} to $M$. Then the union of the traces of orbits in $I_B$ is the image under the locally Lipschitz map of 
evaluation of a $\mathcal{L}_1\times\mathcal{L}_{\{b\}\times N}$-neglectable set. Therefore it is $\mathcal{L}_M$-neglectable.
\end{proof}

The proof of Theorem \ref{Thmmonge2} follows word-by-word the proof of Theorem \ref{Thmmonge} except for obvious changes.

\end{document}